\documentclass[12pt,reqno]{amsart}
\usepackage{amsmath}
\usepackage{amssymb}
\usepackage{amstext}
\usepackage{a4wide}
\usepackage{graphicx}
\usepackage{bm}
\usepackage[numbers,sort&compress]{natbib}
\allowdisplaybreaks \numberwithin{equation}{section}
\usepackage{color}
\usepackage{cases}

\numberwithin{equation}{section}

\newtheorem{theorem}{Theorem}[section]
\newtheorem{proposition}[theorem]{Proposition}
\newtheorem{corollary}[theorem]{Corollary}
\newtheorem{lemma}[theorem]{Lemma}
\newtheorem*{Yudovich's Theorem}{Yudovich's Theorem}

\theoremstyle{definition}

\newtheorem{definition}[theorem]{Definition}

\theoremstyle{remark}
\newtheorem{remark}[theorem]{Remark}

\newcommand{\Om}{\Omega}

\newcommand{\Z}{\mathbb{Z}}
\newcommand{\N}{\mathbb{N}}

\newcommand{\divv}{\operatorname{div}}

\usepackage{hyperref}
\hypersetup{hypertex=true,
	colorlinks=true,
	linkcolor=blue,
	anchorcolor=blue,
	citecolor=blue }
\makeatletter
\def\subsection{\@startsection{subsection}{1}%
	\z@{1\linespacing\@plus\linespacing}{0.1\linespacing}%
	{\normalfont\bfseries}}
\makeatother
\makeatletter
\def\subsubsection{\@startsection{subsubsection}{1}%
	\z@{1\linespacing\@plus\linespacing}{0.1\linespacing}%
	{\normalfont}}
\makeatother
\begin{document}

	\title
	{HELICAL KELVIN WAVES FOR THE 3D EULER EQUATIONS}
	\author{Daomin Cao, Boquan Fan, Rui Li, Guolin Qin}
	    \address{Institute of Applied Mathematics, AMSS, Chinese Academy of Sciences, Beijing 100190, and University of Chinese Academy of Sciences, Beijing 100049,  P.R. China}
        \email{dmcao@amt.ac.cn}
        \address{Institute of Applied Mathematics, AMSS, Chinese Academy of Sciences, Beijing 100190, and University of Chinese Academy of Sciences, Beijing 100049,  P.R. China}
        \email{fanboquan22@mails.ucas.ac.cn}
        \address{Institute of Applied Mathematics, AMSS, Chinese Academy of Sciences, Beijing 100190, and University of Chinese Academy of Sciences, Beijing 100049,  P.R. China}
        \email{lirui2021@amss.ac.cn}
        \address{School of Mathematical Science, Peking University, Beijing 100871, People's Republic of China}
        \email{qinguolin18@mails.ucas.ac.cn}

	
	\begin{abstract}
		
		Helical Kelvin waves were conjectured to exist for the 3D Euler equations in Lucas and Dritschel \cite{LucDri} (as well as in  \cite{Chu}) by studying dispersion	relation for infinitesimal linear perturbations of a circular helically symmetric vortex patch. This paper aims to rigorously establish  the existence of these $m$-fold symmetric  helical Kelvin waves, in both simply and doubly  connected cases, for the 3D Euler equations. The construction is based on linearization of contour dynamics equations and bifurcation theory. Our results rigorously verify the prediction in  aforementioned papers and  extend  $m$-waves of Kelvin from the 2D Euler equations   to the 3D helically symmetric Euler equations. 
	\end{abstract}

	\maketitle
	\section{Introduction}\label{sec1}
     As the oldest mathematical idealization of tip vortices in the wakes behind screws, propellers or wind turbines, helical vortices are commonly observed in  nature and in numerous industrial applications  and have garnered increasing attention of physicists and mathematicians since the pioneering work of Joukowsky in 1912 \cite{Joukowski1912}. Interested readers are referred to \cite{Oku, Wood2020} for reviews on helical vortices. 
     
    In this paper, we shall explore the existence of non-trivial helical vortex patches, which we call \emph{helical Kelvin waves} and which were conjectured to exist  based on the dispersion relation studies conducted by Lucas and Dritschel \cite{LucDri} and Chu and Llewellyn \cite{Chu}. We aim to give a  rigorous proof of existence of such helical vortex patches by using bifurcation theory. Let us begin by introducing the governing equations and providing a historical discussion on helical vortices.
     \subsection{The 3D Euler equations under helical symmetry}
     	We consider the motion of an inviscid incompressible flow  governed by the following Euler equations in $\mathbb R^3$:
     \begin{equation}\label{1-1}
     	\begin{cases}
     		\partial_t\mathbf{v}+(\mathbf{v}\cdot \nabla)\mathbf{v}=-\nabla P,\ \ &\text{in}\ \  \mathbb R^3\times (0,T),\\
     		\nabla\cdot \mathbf{v}=0,\ \ &\text{in}\ \ \mathbb R^3\times (0,T),\\ 
     		\mathbf{v}(\cdot, 0)=\mathbf{v}_0(\cdot), \ \ &\text{in}\ \ \mathbb R^3,
     	\end{cases}
     \end{equation}
     where  $ \mathbf{v}=(v_1,v_2,v_3) $ is the velocity field, $ P$ is the scalar pressure. Let  $\pmb{\omega}:=\nabla\times\mathbf{v}$  be the vorticity of the fluid. By taking curl of the first equation in \eqref{1-1},  we obtain the following equations for vorticity:
     \begin{align}\label{1-2}
     	\begin{cases}
     		\partial_t \pmb{\omega}+(\mathbf{v}\cdot\nabla)\pmb{\omega}=(\pmb{\omega}\cdot\nabla)\mathbf{v},
     		\ \ &\text{in}\ \   \mathbb R^3\times (0,T),\\
     		\mathbf{v}=\nabla\times \pmb\psi,  \quad  -\Delta \pmb\psi=\pmb{\omega},\ \ &\text{in}\ \  \mathbb R^3\times (0,T),\\ 
     		\pmb{\omega}(\cdot, 0)=\nabla\times \mathbf{v}_0(\cdot), \ \ &\text{in}\ \  \mathbb R^3,
     	\end{cases}
     \end{align}
     where \pmb{$\psi$} is the stream function. The above system is usually referred to as the Euler equations in vorticity form. 
     
     We study solutions of the 3D Euler equations with helical symmetry. To be precise, let us define  for any $ \theta\in\mathbb R$ the rotation by an angle $\theta$ around the $x_3$-axis
     \begin{equation}\label{def-of-Rtheta}
     	Q_\theta=\begin{pmatrix}
     		R_\theta & 0  \\
     		0 & 1
     	\end{pmatrix}  \  \ \text{with}\   \ R_\theta=\begin{pmatrix}
     		\cos\theta & \sin\theta  \\
     		-\sin\theta &\cos\theta
     	\end{pmatrix}.
     \end{equation}
     For any helical pitch $ h>0 $, we introduce a group of one-parameter helical transformations $ \mathcal{H}_h=\{H_{\theta}:\mathbb{R}^3\to\mathbb{R}^3\mid  \theta\in\mathbb{R}\} $, where
     \begin{equation*}
     	H_{\theta}(x)=Q_\theta(x)+\begin{pmatrix}
     		0  \\
     		0  \\ 
     		h\theta
     	\end{pmatrix}=\begin{pmatrix}
     		x_1\cos\theta+x_2\sin\theta  \\
     		-x_1\sin\theta +x_2\cos\theta  \\ 
     		x_3+h\theta
     	\end{pmatrix}.
     \end{equation*}
     Following \cite{DDMW2, dutrifoy1999existence, ettinger2009global}, we say that
     a scalar function $f:\mathbb{R}^3\rightarrow\mathbb R $ is  helical, if $$ f(H_{\theta}(x))=f(x),  \ \ \forall\, \theta\in\mathbb{R},\  x\in \mathbb R^3, $$
     and a vector field $ \mathbf{v}:\mathbb{R}^3\rightarrow\mathbb R^3$ is   helical, if $$ \mathbf{v}(H_{\theta}(x))= Q_{\theta} \mathbf{v}(x),\ \ \forall\, \theta\in\mathbb{R},\  x\in \mathbb R^3. $$
     Helical solutions of \eqref{1-1} are then naturally
     defined as follows.
     \begin{definition}\label{HelicalSolution}
     	A pair $\left(\mathbf{v},P\right)$ is a helical solution of \eqref{1-1} if it satisfies \eqref{1-1} and both $\mathbf{v}$ and $P$ are helical  in the above sense.
     \end{definition}
      Assuming, in addition,  that helical solutions satisfy the  \textit{no-swirl condition}:
     \begin{equation}\label{5}
     	\mathbf{v}\cdot\vec{\zeta}=0,
     \end{equation}
     where $\vec{\zeta}=(x_2,-x_1,h)$ is the field of tangents of symmetry lines of $\mathcal{H}_h$. Direct calculations show that, under the assumption \eqref{5}, the vorticity field $\boldsymbol{\omega}$ satisfies
     \begin{equation}\label{1-9}
     	\boldsymbol{\omega}=\frac{\omega_3}{h}\vec{\zeta},
     \end{equation}
     where $\omega_3=\partial_1 v_2-\partial_2 v_1$ is the third component of $\boldsymbol{\omega}$, which is also a helical function (see e.g. \cite{DDMW2, ettinger2009global} for the proof). If we have the value of $\omega_3$ in $\mathbb{R}^2=\mathbb{R}^3\cap\left\{x_3=0\right\}$, we can recover the value in the whole space $\mathbb{R}^3$ by the helical symmetry. In what follows, we may still use $x:=(x_1,x_2)$ to denote points in $\mathbb{R}^2$.  Hence, if we define $\omega(x_1,x_2,t)=\omega_3(x_1,x_2,0,t)$, then $\omega$ was proved in \cite{ettinger2009global} (see also \cite{DDMW2}) to satisfy the following 2D vorticity-stream equations:
     \begin{align}\label{eq1.6}
     	\begin{cases}
     		\partial_t\omega+\nabla^\perp\varphi\cdot\nabla\omega=0,&\text{in}\ \mathbb{R}^2\times (0,T),\\
     		\ \mathcal{L}_H\varphi=\omega,     &\text{in}\ \mathbb{R}^2\times (0,T),\\
     		\omega(\cdot, 0)=\omega_0,&\text{in}\ \mathbb{R}^2,
     	\end{cases}
     \end{align}
     where $\nabla^\perp:=(\partial_{x_2},-\partial_{x_1})^{T}$, $\varphi$ is usually called stream function, and $\mathcal{L}_H(\cdot)=-\divv\left(K_H(x)\nabla\left(\cdot\right)\right)$ is an elliptic operator of divergence type with
     \begin{equation}\label{KH}
     	K_H(x)=\frac{1}{h^2+x_1^2+x_2^2}\begin{pmatrix}
     		h^2+x_2^2&-x_1x_2\\
     		-x_1x_2&h^2+x_1^2
     	\end{pmatrix}.
     \end{equation}
    Conversely, it was also proved in \cite{DDMW2, ettinger2009global} that one can recover a helical symmetric solution to the 3D Euler equations \eqref{1-2} from an arbitrary solution of \eqref{eq1.6}. 
     
    In recent years, there has been tremendous interest in the mathematical study of helical vortices. The existence and uniqueness of weak solutions to Cauchy problem for \eqref{eq1.6} were initially proved by Ettinger  and Titi \cite{ettinger2009global}. Recent progress on this topic can also be found in \cite{guo2025global,dutrifoy1999existence,bronzi2015global,jiu2017global, ABIDI20162177}. On the other hand, various families of extremely concentrated traveling-rotating helical vortices for \eqref{eq1.6} were constructed in e.g. \cite{CWmathann, DDMW2, GM2024,cao2024concentrated,cao2023helical,cao2023helical2,cao2023desingularization} by desingularizing helical vortex filaments via the gluing method or various variational methods.
     
     \bigskip
     
    However, to the best of our knowledge, there is no literature on the existence of non-trivial, non-concentrated helical vortices. In the work of Lucas and Dritschel \cite{LucDri} (as well as  Chu and Llewellyn \cite{Chu}), the authors investigated the dispersion relation for infinitesimal linear perturbations of a trivial solution and conjectured the existence of uniformly rotating  $m$-fold symmetric helical vortex patches. In this paper, we aim to rigorously establish the existence of uniformly rotating   $m$-fold symmetric helical vortex patches, in both simply   and doubly connected cases, for \eqref{eq1.6}.
    
    Note that uniformly rotating $m$-fold symmetric vortex patches for the 2D Euler equations, known as Kelvin waves,  have been extensively studied.  Before presenting our results, let us briefly review the relevant research on Kelvin waves  for the 2D Euler equations.

	\subsection{Kelvin waves for the 2D Euler equations}\label{sec1.2}
Recall the 2D Euler equations in  the vorticity formulation:
	\begin{align}\label{1.2}
	\begin{cases}
	\partial_t\omega+\textbf{v}\cdot\nabla\omega=0,&\text{in}\ \mathbb{R}^2\times (0,T),\\
		\ \textbf{v}=\nabla^\perp(-\Delta)^{-1}\omega,&\text{in}\ \mathbb{R}^2\times (0,T),
	\end{cases}
\end{align}
where $\omega=\partial_{x_1}v_2-\partial_{x_2}v_1$ is the vorticity.  A few coherent vortices have been found, such as rotating disk, sliding dipoles, etc. Among them, we revisit the $m$-waves of Kelvin, which are uniformly rotating patch solutions of \eqref{1.2}. For any integer $m\geq 2$ and positive number $a>0$, the $m$-waves of Kelvin can be expressed in the following form
\begin{equation}\label{1.3}
	\omega^{m,s}=1_{D^{m,s}},\ \   D^{m,s}=\left\{\rho e^{i\phi}:\rho<a+\zeta^{m,s}(\phi)\right\}
\end{equation}
where 
\begin{equation*}
	\zeta^{m,s}(\phi)=s\cos(m\phi)+o(s).
\end{equation*}
Moreover, there exists a $\Omega^{m,s}\in\mathbb{R}$ such that 
\begin{equation*}
	\omega(x,t)=\omega^{m,s}(R_{-\Omega^{m,s}t}x)
\end{equation*}
solves the 2D Euler equations \eqref{1.2} weakly, where the rotation $R_{-\Om^{m,s}t}$ is defined in \eqref{def-of-Rtheta}.

The first example for such solutions was due to Kirchhoff dating back to 1894. He demonstrated that ellipses represent uniformly rotating patch solutions for any aspect ratio, as discussed in \cite{kirchhoff1894vorlesungen}, corresponding to the case where $m=2$. For $m\geq 3$, Kelvin first suggested in 1880 the possibility of the existence of $m$-fold symmetric rotating patches bifurcating from a disc (see e.g. Lamb \cite{lamb1924hydrodynamics}), calculating that an infinitesimal perturbation of the disc with period $m$ would rotate with an angular velocity of $ \frac{m-1}{2m}$. In 1982, Burbea \cite{burbea1982motions} provided an argument showing the existence of such solutions, referring to them as ``$m$-waves of Kelvin", which was later rigorously established by Hmidi, Mateu, and Verdera in \cite{hmidi2013boundary} in 2013.  Moreover, the authors also proved the $C^\infty$-regularity for boundaries of the $m$-waves of Kelvin in  \cite{hmidi2013boundary}. Uniformly rotating analytic patches were then constructed in \cite{CCG} for both the 2D Euler equations and the gSQG equations. Subsequently, by applying bifurcation theory, doubly connected $m$-fold symmetric vortex patches for the planar Euler equations bifurcating from an annulus were shown to exist in \cite{dHMV2016, HMIDI2016799, wang2024degenerate}.   Existence and nonexistence of $m$-fold symmetric vortex patches for Euler equations in the disk were investigated in \cite{dHHM2016, WANG2021124695}. Further exploration of the full branch of solutions can be found in the paper of Hassainia, Masmoudi, and Wheeler \cite{hassainia2020global}. For $m$-fold symmetric patches of other models, interested readers may refer to \cite{dHH2016, DHR2019, hassainia2015v, hmidi2024uniformly, GOMEZSERRANO2019110, GHM2022, garcia2023time,castro2016existence, Emeric, HMIDI2023110142, garcia2023caps} and references therein. We highlight the papers \cite{hmidi2013boundary, hmidi2024uniformly}, which inspired some arguments in this work. We would also like to refer  readers to \cite{symmetry, PARK2022108779} for nonexistence of nontrivial rotating vortices. Stability and instability for Kelvin waves were recently studied in \cite{stabilty}.

\subsection{Main results and ideas of proofs}
Now, we return to the main topic of this paper: the construction of non-trivial  uniformly rotating  $m$-fold symmetric  patch-type solutions for \eqref{eq1.6}. Specifically, we are going to find solutions that we shall call ``helical $m$-waves of Kelvin" (or simply, ``helical Kelvin waves") defined as follows.
\begin{definition}\label{DefHKV}
	Given an integer $m\geq 2$, a domain $D$ and a function $\omega(x,t)=1_D(R_{-\Omega t}x)$, where $R_{-\Om t}$ is defined in \eqref{def-of-Rtheta},  we shall call $\omega$   a helical $m$-wave  of Kelvin if the domain $D$ is $m$-fold symmetric and $\omega$ solves  equations \eqref{eq1.6} in the sense of distribution.
\end{definition}

Being different from the 2D Euler equations, the expression of the operator  $\mathcal L_H$  in equation  \eqref{eq1.6} appears quite complicated, which makes the existence of trivial solutions not obvious. In order to introduce trivial solutions as our starting points of bifurcation, we need the following expression in polar coordinates $(\rho, \phi)$ of the operator $\mathcal L_H$ appearing in \eqref{eq1.6}. 
\begin{equation}\label{Lpolar}
	\mathcal L_H=-\frac{h^2 }{\rho}\partial_\rho\left(\frac{ \rho}{h^2+\rho^2} \partial_\rho\right)-\frac{1}{\rho^2}\partial_\phi^2.
\end{equation} 
It can be seen that the operator $\mathcal L_H$ is invariant under translations in the angle  $\phi$. In other words, 	 $\mathcal L_H$  is invariant under rotations  around the origin in  Cartesian coordinates. In view of the transport nature of the first equation in  \eqref{eq1.6},   the rotational invariance of $\mathcal L_H$ implies that any function that is radially symmetric around the origin  constitutes a trivial solution of  \eqref{eq1.6}. This  assertion  will be checked later in Corollary \ref{trivialsolution}. It is important to note that, in contrast to the 2D Euler equations, functions that are radially symmetric around a point other than the origin may not yield trivial solutions for \eqref{eq1.6}, as the matrix $K_H$ defined in \eqref{KH} is not invariant under translations.

By carefully studying the linearization of contour dynamics equations for patches of  \eqref{eq1.6} near a given disk centered at the origin and applying the bifurcation theory,  we establish our first main result on the existence of time-periodic helical  $m$-waves of Kelvin.
\begin{theorem}\label{thm-simply}
	Let $a>0, h>0$ be two numbers and $m\geq 2$ be an integer. Then there exist three constants $M(a,h)$, $H_0(a)$ and $H_1(a)$ such that if either  
$m>M(a,h)$ or $h\in (0, H_0(a))\cup (H_1(a), +\infty)$ 
, then there exists a curve of simply connected helical $m$-waves of Kelvin $\omega_H$ for \eqref{eq1.6} bifurcating from the trivial solution $$\omega_{trivial}= {1}_{B_a(0)}$$ at the angular velocity
	 $$\Om_m=\Om_m(a,h):=\frac{a^2}{h^2}I'_{m}\left(\frac{ma}{h}\right)K'_{m}\left(\frac{ma}{h}\right)+\frac{h^2+a^2}{2h^2}.$$
	 Here $I_m$ and $K_m$ are the modified Bessel functions of the first and second kind, of order $m$, and the superscript `` prime '' denotes the derivative with respect to the argument. 
	 
	 Moreover, there exists $s_0>0$ such that for each $s\in (-s_0, s_0)$, there exist a simply connected $m$-fold symmetric domain $D_H^{m,s}$ and a real number $\Omega_H^{m,s}=\Om_m+o(1)$ such  that
	 \begin{equation*}
	 	\omega_H(x,t)=1_{D_H^{m,s}}(R_{-\Omega_H^{m,s}t}x)
	 \end{equation*}
 and 
 \begin{equation*}
 	D_H^{m,s}=\left\{\rho e^{i\phi}:\rho<a+s\cos(m\phi)+o(s)\right\}.
 \end{equation*}
 
\end{theorem}
Some comments about our results in Theorem \ref{thm-simply} are listed as follows.
\begin{remark}
	Our result rigorously establishes the existence of helical Kelvin waves for the 3D Euler equations, thereby  verifying  the  conjecture in Lucas and Dritschel \cite{LucDri} and Chu and Llewellyn \cite{Chu}. Note that our formulae for $\Om_m$ are slightly different from the one in \cite{Chu, LucDri}, since we do not assume the presence of vortex sheets at the boundaries of patches.
\end{remark}
\begin{remark}
	 Formally speaking, the helical flow governed by \eqref{eq1.6} approaches to planar flow as the helical pitch $h$ tends to $+\infty$. This can be seen from the fact that the operator $\mathcal L_H$ will formally tend to $-\Delta$ in this limit, which implies that equations  \eqref{eq1.6} tend  to the 2D Euler equations.  
	 
	 On the other hand, by  employing  expansions \eqref{bzk} for  modified Bessel functions, one can check that the critical angular velocity in Theorem \ref{thm-simply} satisfies
	 $$\Om_m\to \frac{m-1}{2m}, \ \ \ \text{as}\ h\to+\infty.$$
	 It is noteworthy that the   $m$-waves of Kelvin  for the 2D Euler equations were obtained in e.g. \cite{hmidi2013boundary} by bifurcating from disks with  angular velocity $ \frac{m-1}{2m}$. Consequently,  our result in Theorem \ref{thm-simply}  can be viewed as a generalization  of $m$-waves of Kelvin from the 2D Euler equations   to the 3D helically symmetric Euler equations.
\end{remark}
\begin{remark}
	 The boundaries of helical Kelvin waves constructed in Theorem \ref{thm-simply}  exhibit $C^{1,\frac{1}{2}}$ regularity.   By adapting arguments in works such as  \cite{CCG, hmidi2013boundary, GH}, one may improve the regularity of these boundaries to be analytic.
\end{remark}
\begin{remark}
	  Numerical computations show that the assumptions that either  
	  $m>M(a,h)$ or $h\in (0, H_0(a))\cup (H_1(a), +\infty)$  in Theorem \ref{thm-simply} can be  removed, as discussed in Remark \ref{rem} below. However, a rigorous proof requires more detailed analysis.  
\end{remark}

The proof of Theorem \ref{thm-simply} is accomplished by applying bifurcation theory of Crandall and Rabinowitz  \cite{CRANDALL1971321} (see also Theorem \ref{CR} in the appendix of this paper) to the contour dynamics equations for patch-type solutions for \eqref{eq1.6}, which we write as  $$\mathcal{F}(\Omega,r)=0$$ for certain functional $\mathcal{F}$.

The strategy is to demonstrate that the functional $\mathcal{F}$ meets all the assumptions required by Crandall-Rabinowitz Theorem \ref{CR}, which primarily involves the regularity analysis of these nonlinear functionals and the study of spectral of their linearized operators around the trivial solutions $(\Om, 0)$. Unlike previous works such as \cite{burbea1982motions,castro2016existence,hassainia2015v,hmidi2013boundary,hmidi2024uniformly}, where the kernels involved in the functionals are treated differently, here they are represented by Fourier series (see section \ref{sec2}) with precise Fourier coefficients. This formulation facilitates the accurate computation of the angular velocity but complicates the proof of regularity for the functional $\mathcal{F}$. To address this challenge, we note that the singularity of the Green function is locally the same in both the whole space and bounded domains. Thus, the Green function in bounded domains can be effectively utilized to verify the regularity of the functional. The above approach is based on the results presented in \cite{cao2024expansion}, where an expansion for the Green function of general divergence-type elliptic operators in bounded domains is obtained. In this work, the degree of singularity is clearly shown to be approximately logarithmic. In the spectral analysis, the study will involve the product of derivatives of modified Bessel functions of the first and second kinds (i.e.,  $I'_{m}\left(\frac{ma}{h}\right)K'_{m}\left(\frac{ma}{h}\right)$). In order to derive the one-dimensional properties requested in Crandall-Rabinowitz bifurcation  theorem, the most essential and difficult step is to verify the monotonicity of the product $I'_{m}\left(\frac{ma}{h}\right)K'_{m}\left(\frac{ma}{h}\right)$ in $m$.  Our key idea is that by using some estimates in \cite{segura2011bounds,segura2021monotonicity} for modified Bessel functions, we are able to  reduce the checking of monotonic property into the problem  of verifying the fact that a scalar function with a specific expression is greater than zero; see Proposition \ref{prop3-6} for details. Using this reduction, through a series of refined estimates, we successfully  prove the monotonicity of the product of these derivatives under very weak conditions and provide decay estimates. These results are sufficient to verify Theorem \ref{CR}, thereby completing the proof.

\bigskip

Our second main result concerns the existence of doubly connected time-periodic helical patches for  \eqref{eq1.6}.
\begin{theorem}\label{thm-doubly}
	Let $a_1>a_2>0, h>0$ be three numbers and $m\geq 2$ be an integer. Then there exists a constant $M_h(a_1, a_2)>0$ such that if $m>M_h(a_1, a_2)$, there exist two curves of doubly connected   $m$-fold symmetric vortex patches $\hat{\omega}_H^{m,s, \pm}$ for \eqref{eq1.6} bifurcating from the trivial solution $$\hat{\omega}_{trivial}= {1}_{B_{a_1}(0)\setminus \overline{B_{a_2}(0)}}$$ at the angular velocity
	\begin{equation*}
		\begin{split}
			\Om_m^{\pm}=\frac{(h^2+a_1^2)(a_1^2-a_2^2)}{4h^2 a_1^2}+\frac{a_1^2}{2h^2}I'_{m}\left(\frac{ma_1}{h}\right)K'_{m}\left(\frac{ma_1}{h}\right)-\frac{a_2^2}{2h^2}I'_{m}\left(\frac{ma_2}{h}\right)K'_{m}\left(\frac{ma_2}{h}\right) \pm \frac{\sqrt{\triangle_m}}{2},
		\end{split}
	\end{equation*}
	where 	
	\begin{equation*}
		\begin{split}
		\triangle_m=& \left(\frac{(h^2+a^2)(a_1^2-a_2^2)}{2h^2 a_1^2}+\frac{a_1^2}{h^2}I'_{m}\left(\frac{ma_1}{h}\right)K'_{m}\left(\frac{ma_1}{h}\right)+ \frac{a_2^2}{h^2}I'_{m}\left(\frac{ma_2}{h}\right)K'_{m}\left(\frac{ma_2}{h}\right)   \right)^2\\
		&- 4\left(\frac{a_1a_2}{ h^2}I'_{m}\left(\frac{ma_2}{h}\right)K'_{m}\left(\frac{ma_1}{h}\right)\right)^2.	
		\end{split}
	\end{equation*}		
Moreover, there exists $s_1>0$ such that for each $s\in (-s_1, s_1)$, there exist two pairs of simply connected connected m-fold symmetric domain $\hat{D}_{H,i}^{m,s, \pm}$ for $i=1,2$ and two  real numbers $\hat{\Omega}_H^{m,s, \pm}=\Om_m^{\pm}+o(1)$ such  that
 \begin{equation*}
	\hat{D}_{H,2}^{m,s, \pm}\subset \hat{D}_{H,1}^{m,s, \pm},
\end{equation*}
	 \begin{equation*}
\hat{\omega}_H^{m,s, \pm}(x,t)=1_{\hat{D}_{H,1}^{m,s, \pm}\setminus \overline{\hat{D}_{H,2}^{m,s, \pm}}}(R_{-\hat{\Omega}_H^{m,s, \pm}t}x),
\end{equation*}
and the two domain can be written as
 \begin{equation*}
	\hat{D}_{H,1}^{m,s, \pm}=\left\{\rho e^{i\phi}:\rho<a_1+s\left(\Om_m^\pm+\frac{a_2^2}{h^2}I'_{m}\left(\frac{m  a_2}{h}\right)K'_{m}\left(\frac{m a_2 }{h}\right)\right)\cos(m\phi)+o(s)\right\},
\end{equation*}
\begin{equation*}
	\hat{D}_{H,2}^{m,s, \pm}=\left\{\rho e^{i\phi}:\rho<a_2-s\left(\frac{a_1a_2}{h^2}I'_{m}\left(\frac{m  a_2}{h}\right)K'_{m}\left(\frac{m a_1 }{h}\right)\right)\cos(m\phi)+o(s)\right\}.
\end{equation*}
\begin{remark}
	If we set $a_1=1$ and $a_2=b$ for some $0<b<1$, then as $h\to\infty$, we have
	\begin{equation*}
		\Om_m^{\pm}\to \frac{1-b^2}{4}\pm\frac{1}{2m}\sqrt{\left(\frac{(1-b^2)m}{2}-1\right)^2-b^{2m}}.
	\end{equation*}
 It is noteworthy that the  doubly connected $m$-waves of Kelvin  for the 2D Euler equations were obtained, for instance,  \cite{dHMV2016} by bifurcating from annulus with  angular velocity $\frac{1-b^2}{4}\pm\frac{1}{2m}\sqrt{\left(\frac{(1-b^2)m}{2}-1\right)^2-b^{2m}}$. Consequently,  our result in Theorem \ref{thm-doubly}  can be viewed as a generalization  of doubly connected $m$-waves of Kelvin from the 2D Euler equations   to the 3D helically symmetric Euler equations.
\end{remark}
\end{theorem}
\subsection{Notation and organization}
In this subsection, we first introduce some notations. We are going to use the standard notations for Lebesgue, Sobolev and Hölder spaces: $L^p(D)$, $W^{k,p}(D)$, $C^\alpha(D)$..., where $p\in[1,\infty)$, $k\in\mathbb{N}^*$, $\alpha\in(0,1)$ and $D$ is a domain in $\mathbb{R}^2$. We will use ``$\lesssim$'' to denote the statement ``$\leq C$'' and ``$\lesssim_\delta$'' to denote the statement ``$\leq C_\delta$''.

We end this section by outlining organization of this paper.
This paper is structured into sections based on the Green function expansion of the operator $\mathcal{L}_H$ and the cases of simply connected and doubly connected patches separately. Section \ref{sec2} focuses on presenting the Green’s function of the operator $\mathcal{L}_H$ and studying its properties. Section \ref{sec3} is dedicated to constructing simply connected Kelvin waves, while Section \ref{sec4} is devoted to constructing doubly connected Kelvin waves. In Appendix \ref{apa}, we provide the statement of the Crandall-Rabinowitz bifurcation theorem, while in Appendix \ref{apb}, we present some Hölder estimates for certain integral operators.
\section{Preliminary: the Green function for \texorpdfstring{$\mathcal L_H$}{} and trivial solutions} \label{sec2}
	In this section, we will collect and prove some basic properties for the Green function associated with the elliptic operator  $\mathcal L_H=-\divv (K_H\nabla (\cdot)) $ with $K_H$ defined in \eqref{KH}. 
 
	It is useful to derive an expression of the operator $\mathcal{L}_H$ in terms of polar coordinates. Denoting $(x_1,x_2)=\rho e^{i\phi}$ and $\kappa=1/h$. Direct computation gives the following expression (see \cite{DDMW2, GM2024})
	\begin{equation}\label{eq2.2}
		\mathcal{L}_H\Psi=\mathcal{L}_{\rho,\phi}\Psi:=-\frac{1}{\rho}\partial_\rho\left(\frac{\rho}{1+\kappa^2\rho^2}\partial_\rho\Psi\right)-\frac{1}{\rho^2}\partial_\phi^2\Psi\quad {\rm in}\ \Pi,
	\end{equation}
	where $\Pi=\{(\rho,\phi):\rho>0,\phi\in\mathbb{T}\}$. 
	In order to invert $\mathcal{L}_H\Psi=w$ by
	\begin{equation}\label{integral}
		\Psi=\mathcal G_Kw:=\int_{\mathbb R^2} G_H(x,y) w(y) dy=\int\int_{\Pi} G_H(\rho,\phi;\rho_0,\phi_0)w(\rho_0,\phi_0)\rho_0\,d\rho_0\,d\phi_0,
	\end{equation}
	the Green function $G_H(r,\phi;r_0,\phi_0)$ should satisfy
	\begin{equation}\label{eq2.4}
		\mathcal{L}_{r,\phi}G_H=\rho_0^{-1}\delta_{(\rho_0,\phi_0)}.
	\end{equation}
	Using a Fourier expansion in the variable $\phi$ leads to a decomposition of $G_H$ in the form
	\begin{equation}\label{Fourier}
		G_H(\rho,\phi;\rho_0,\phi_0)=-\sum_{m\in\mathbb{Z}}G_m(\rho,\rho_0)e^{im(\phi-\phi_0)}.
	\end{equation}
	Then \eqref{eq2.2} and \eqref{eq2.4} lead to the ordinary differential equations
	\begin{equation*}
		\frac{\rho}{1+\kappa^2\rho^2}G_m^{''}+\frac{1-\kappa^2\rho^2}{(1+\kappa^2\rho^2)^2}G_m^{'}-\frac{m^2}{\rho}G_m=\frac{\delta(\rho-\rho_0)}{2\pi},\quad \text{for}\ m\in\mathbb{Z}.
	\end{equation*}
	It can be shown by solving the above ODEs that, for $m\neq 0$, we have
	\begin{align}\label{2.5}
		G_m(\rho,\rho_0)=\frac{\rho\rho_0}{2\pi h^2}
		\begin{cases}
			I_m'\left(\frac{m\rho}{h}\right)K_m'\left(\frac{m\rho_0}{h}\right),&\text{if} ~\rho<\rho_0,\vspace{0.25em}\\
			K_m'\left(\frac{m\rho}{h}\right)I_m'\left(\frac{m\rho_0}{h}\right),&\text{if} ~\rho>\rho_0,
		\end{cases}
	\end{align}
	and for $m=0$, we have
	\begin{align}\label{2.6}
		G_0(\rho,\rho_0)=\begin{cases}
			0,&\text{if} ~\rho<\rho_0,\\
			\frac{1}{2\pi}\int_{\rho_0}^{\rho}\frac{1+\kappa^2s^2}{s}ds,&\text{if} ~\rho>\rho_0.
		\end{cases}
	\end{align}
	Here $I_m$ and $K_m$ are the modified Bessel functions of the first and second kind, of order $m$, and the superscript prime denotes a derivative with respect to the argument.
	
	Before verifying \eqref{integral},  we present some useful estimates for the modified Bessel functions $I_\nu$ and $K_\nu$ developed in \cite{segura2011bounds,segura2021monotonicity}, which will be utilized in our subsequent proofs.
	\begin{lemma}\label{lemma3.4}
		For $z>0$, denote 
		\begin{equation}
			C(I_\nu(z))=z\frac{I'_\nu(z)}{I_\nu(z)},\quad C(K_\nu(z))=-z\frac{K'_\nu(z)}{K_\nu(z)},
		\end{equation}
		which are positive quantities if $\nu\geq 0$ ( $C(K_\nu(z))$ is positive for all $\nu$). There holds
		\begin{equation}
			\sqrt{(\nu+1)^2+z^2}-1<C(I_\nu(z))<\sqrt{(\nu+\frac{1}{2})^2+z^2}-\frac{1}{2},
		\end{equation}
		where the lower bound holds for $\nu\geq -1$ and the upper bound holds for $\nu\geq -\frac{1}{2}$. 
		
		Moreover, we have 
		\begin{equation}
			\sqrt{(\nu-\frac{1}{2})^2+z^2}+\frac{1}{2}<C(K_\nu(z))<\sqrt{(\nu-1)^2+z^2}+1,\quad {\rm for}\  \nu\geq 1.
		\end{equation}
	\end{lemma}
	\begin{lemma}\label{lemma3.5}
		Let $z>0$, the following two bounds hold:
		\begin{equation}
			\begin{split}
				I_\nu(z)K_\nu(z)<\frac{1}{2\sqrt{(\nu-\frac{1}{2})^2+z^2}},\quad&{\rm for}\ \nu\geq \frac{1}{2} ,\\
				I_\nu(z)K_\nu(z)>\frac{1}{1+\sqrt{\nu^2+z^2}+\sqrt{(\nu-1)^2+z^2}},\quad&{\rm for}\ \nu\geq -1.
			\end{split}
		\end{equation}
	\end{lemma}

	It is worth noting that the Green function expansion in \eqref{Fourier} has been widely used in physics literature (see for example \cite{Chu,LucDri}). However, to the best of the author’s knowledge, no rigorous proof exists for this Green’s function expansion from a mathematical standpoint. The following we provide a rigorous proof of this expansion for completeness.
	\begin{proposition}
		The coefficient $\{G_m\}_{m\in\Z}$ defined in \eqref{2.5}-\eqref{2.6} satisfies that, for any $\varphi\in C_c^\infty(0,+\infty)$ and $\rho_0\in(0,\infty)$, it holds
		\begin{equation}\label{G_m equation}
			\int_{0}^\infty  \frac{\rho}{1+\kappa^2 \rho^2}\partial_\rho G_m(\rho,\rho_0)\partial_\rho \varphi(\rho)+\frac{m^2}{\rho}G_m(\rho,\rho_0)\varphi(r)\,dr=-\frac{\varphi(\rho_0)}{2\pi}.
		\end{equation}  
		Furthermore, for any $f\in L^\infty((0,\infty)\times (0,2\pi))$ with compact support, the function
		\begin{equation}
			\mathcal{G}_Hf(\rho,\phi)=\iint_{\Pi}G_H(\rho,\phi;\rho_0,\phi_0)f(\rho_0,\phi_0)\rho_0\,d\rho_0\,d\phi_0
		\end{equation}   
		is well-defined, and satisfies $\mathcal{L}_{\rho, \phi}(\mathcal{G}_H f)=f$ in the following weak sense:
		\begin{equation}
			\iint_{\Pi}\left\{\frac{\rho}{1+\kappa^2 \rho^2}\partial_\rho (\mathcal{G}_K f)\partial_\rho \psi(\rho,\phi)-\frac{1}{\rho}\mathcal{G}_K f(\rho,\phi)\partial_\phi^2\psi(\rho,\phi)\right\}d\rho d\phi=\iint_{\Pi}f(\rho,\phi)\psi(\rho,\phi)\rho d\rho d\phi,
		\end{equation} 
		for any $\psi\in C_c^\infty(\Pi).$
	\end{proposition}
	\begin{proof}
		Due to Lemma \ref{lemma3.4}, we can rewrite the coefficient $G_m(r,r_0)$ as follows
		\begin{equation}
			G_m(\rho,\rho_0)=-\frac{1}{2\pi m^2}
			\begin{cases}
				C\left(I_m\left(\frac{m\rho}{h}\right)\right)C\left(K_m\left(\frac{m\rho_0}{h}\right)\right)I_m\left(\frac{m\rho}{h}\right)K_m\left(\frac{m\rho_0}{h}\right),&{\rm if}\ \rho<\rho_0\\
				C\left(I_m\left(\frac{m\rho_0}{h}\right)\right)C\left(K_m\left(\frac{m\rho}{h}\right)\right)I_m\left(\frac{m\rho_0}{h}\right)K_m\left(\frac{m\rho}{h}\right),&{\rm if}\ \rho>\rho_0  
			\end{cases}
		\end{equation}
		Then thanks to the following properties of modified Bessel function $I_m$ and $K_m$ (see \cite{bessel-2}):
		\begin{equation}
			\begin{split}
				&I_m''(z)+\frac{1}{z}I'_m(z)=\left(1+\frac{m^2}{z^2}\right)I_m(z),\\
				&K_m''(z)+\frac{1}{z}K'_m(z)=\left(1+\frac{m^2}{z^2}\right)K_m(z),\\
				&I_m(z)K_m'(z)-K_m(z)I_m'(z)=-\frac{1}{z},
			\end{split}
		\end{equation}
		they imply that \eqref{G_m equation} holds and we have
		\begin{align}
			\partial_\rho G_m(\rho,\rho_0)=&\frac{m \rho_0}{2\pi h}\frac{1+\kappa^2\rho^2}{\rho}
			\begin{cases}
				K_m'\left(\frac{m \rho_0}{h}\right)I_m\left(\frac{m \rho}{h}\right),&{\rm if}\ \rho<\rho_0,\\
				I_m'\left(\frac{m \rho_0}{h}\right)K_m\left(\frac{m \rho}{h}\right),&{\rm if}\ \rho>\rho_0,
			\end{cases}\\
			=&\frac{1}{2\pi m}\frac{1+\kappa^2\rho^2}{\rho}
			\begin{cases}
				-C\left(K_m\left(\frac{m\rho_0}{h}\right)\right)K_m\left(\frac{m\rho_0}{h}\right)I_m\left(\frac{m\rho}{h}\right),&{\rm if}\ \rho<\rho_0,\\
				C\left(I_m\left(\frac{m\rho_0}{h}\right)\right)I_m\left(\frac{m\rho_0}{h}\right)K_m\left(\frac{m\rho}{h}\right),&{\rm if}\ \rho>\rho_0.\\
			\end{cases}
		\end{align}    
		For $m\geq0$, we have the estimates of the upper bounds for $C(I_m(z))$ and $C(K_m(z))$, by using Lemma \ref{lemma3.4},
		\begin{equation}
			\begin{split}
				C(I_m(z))\leq \sqrt{4m^2+z^2},\\
				C(K_m(z))\leq 2\sqrt{m^2+z^2},
			\end{split}
		\end{equation}
		and the expressions of $I_m(z)$ and $K_m(z)$ (see \cite{bessel-2})
		\begin{align}\label{bzk}
			I_m(z)=&\frac{1}{m!}\left(\frac{z}{2}\right)^m\sum_{k=0}^\infty \frac{\left(\frac{z^2}{4}\right)^km!}{k!(m+k)!}\leq \left(\frac{z}{2}\right)^m\frac{e^{\frac{z^2}{4}}}{m!},
		\end{align}
		\begin{align*}	
			K_m(z)=&\frac{(m-1)!}{2}\left(\frac{z}{2}\right)^{-m}\sum_{k=0}^{m-1}\frac{(m-k-1)!}{(m-1)!k!}\left(-\frac{z^2}{4}\right)^k-\ln\left(\frac{z}{2}\right)\frac{1}{m!}\left(-\frac{z}{2}\right)^m\sum_{k=0}^\infty\frac{\left(\frac{z^2}{4}\right)^km!}{k!(m+k)!}\\
			&+\frac{1}{2m!}\left(-\frac{z}{2}\right)^m\sum_{k=0}^\infty \left(\psi(k+1)+\psi(m+k+1)\right)\frac{\left(\frac{z^2}{4}\right)^km!}{k!(m+k)!},
		\end{align*}
		where $\psi(l)=\sum_{s=1}^l\frac{1}{s}-\gamma$, $\gamma=\Gamma'(1)$ and $\Gamma(\tau)=\int_0^\infty t^{\tau-1}e^{-t}\,dt$ is the Gamma function. 
		Combining the fact $I_m(z)=I_{-m}(z)$ and $K_m(z)=K_{-m}(z)$ for all $m\in \N$ and above estimates and expressions, it shows that the series of functions is uniformly convergent, 
		which allows us to change the order of differentiation, integration and summation, i.e. 
		\begin{equation}
			\begin{split}
				&\mathcal{G}_H f(\rho,\phi)=-\sum_{m\in\mathbb{Z}}\iint_{\Pi}G_m(\rho,\rho_0)e^{im(\phi-\phi_0)}f(\rho_0,\phi_0)\rho_0\,d\rho_0\,d\phi_0,\\
				&\partial_\rho \mathcal{G}_Hf(\rho,\phi)=-\sum_{m\in\Z}\iint_{\Pi}\partial_\rho G_m(\rho,\rho_0)e^{im(\phi-\phi_0)}f(\rho_0,\phi_0)\rho_0\,d\rho_0\,d\phi_0.
			\end{split}
		\end{equation}
		Due to Fubini's theorem, we have, for any $\psi\in C_c^\infty$
		\begin{align*}
			&\iint_{\Pi}\left\{\frac{\rho}{1+\kappa^2 \rho^2}\partial_\rho (\mathcal{G}_H f)\partial_\rho \psi(\rho,\phi)-\frac{1}{\rho}\mathcal{G}_H f(\rho,\phi)\partial_\phi^2\psi(\rho,\phi)\right\}\,d\rho\,d\phi\\
			&=-\sum_{m\in\Z}\iint_{\Pi}f(\rho_0,\phi_0)\rho_0\left\{\int_0^{2\pi}e^{im(\phi-\phi_0)}\int_0^\infty\frac{\rho\partial_\rho G_m(\rho,\rho_0)\partial_\rho \psi}{1+\kappa^2\rho^2}+\frac{m^2G_m(\rho,\rho_0)\psi}{\rho}d\rho d\phi\right\}\\
			&=\sum_{m\in\Z}\iint_{\Pi}\left(\int_0^{2\pi}\frac{\psi(\rho_0,\phi)}{2\pi}e^{im\phi}\,d\phi\right)e^{-im\phi_0}f(\rho_0,\phi_0)\rho_0\,d\rho_0\,d\phi_0\\
			&=\iint_\Pi \sum_{m\in\Z}\psi_{-m}(\rho_0)e^{-im\phi_0}f(\rho_0,\phi_0)\rho_0\,d\rho_0\,d\phi_0\\
			&=\iint_{\Pi}\psi(\rho_0,\phi_0)f(\rho_0,\phi_0)\rho_0\,d\rho_0\,d\phi_0,
		\end{align*}
		which completes the proof.
	\end{proof}
	
	As a corollary of the Fourier expansion for the Green function, we obtain a family of trivial solutions for \eqref{eq1.6}.
	\begin{corollary}\label{trivialsolution}
		Let $\omega_0$ be a function radially symmetric around the origin. Then $\omega(x,t)=\omega_0(R_{-\Om t}x)$ solves \eqref{eq1.6} for any $\Om\in \mathbb R$.  
	\end{corollary}
	From the above corollary, we are able to give some typical examples for trivial solutions,  including the (helical) Rankine vortex $\omega_0=1_{B_a(0)} $ as well as the annulus $\omega_0=1_{B_{a_1}(0)\setminus \overline{B_{a_2}(0)}} $. 
	
	One can certainly study the singularity of the Green function $G_H(x,y)$ as $y$ approaches $x$ by the Fourier expansion in \eqref{Fourier}. However, we prefer to use    the following  Green's expansion in bounded domains developed in \cite{cao2024expansion} which was obtained by regularity theory for second order uniformly elliptic differential equations.  
	\begin{theorem}[Theorem 1.4 in \cite{cao2024expansion}]\label{decom}
		Let $\alpha \in(0,1)$ be an arbitrary constant and $G_{H,D}$ be Green's function for $-\divv(K_H\nabla(\cdot))$ in the bounded domain $D$ with zero Dirichlet boundary value. Then there exists a unique $\Phi_i\in F_{i+2+2(2i-1)}^{4i}$ for $i=1,2$ depending on $y\in D$ and function $H_2( x, y)=H_{2,y}( x)$, where $H_2\in C^2(D\times D)$ and $H_{2,(\cdot)}\in C(D,C^{2,\alpha}(\Bar{D}))$  such that
		\begin{equation}\label{2-15}
			G_{H,D}( x, y)=-\frac{\sqrt{{\rm det} K_H( y)}^{-1}}{2\pi}\ln|T_ y( x- y)|+\sum_{i=1}^2\Phi_i(T_ y( x- y))+H_2( x, y)\quad {\rm in}\ \Bar{D}\times D,
		\end{equation}
		where $T_y$ is the matrix given by the Cholesky decomposition i.e.
		$T_y^{-1}(T_y^{-1})^{t}=K_H(y)$ for all $y\in D$.
	\end{theorem}
	\begin{remark}\label{re2-4}
		The sets in above are defined by 
		\begin{equation*}
			F_{k+2m}^{2+2m}:=\left({\rm span}\left\{\frac{ x^\alpha}{| x|^{k+2m}}:|\alpha|=2+2m\right\}\setminus\mathbb{R}[ x]\right)\oplus {\rm span}\left\{ x^\alpha\ln| x|:|\alpha|=k-2,\ \alpha\in\mathbb{N}^2\right\}, 
		\end{equation*}where $\mathbb{R}[x]$ is the set of real polynomials with variables $x_i$, hence the most singular term in $G_{H,D}$ is approximately logarithmic.
	\end{remark}
	The above  Green's expansion in bounded domains plays a central role in  our later analysis of regularity for the functional appearing in the contour dynamics equation. Our key observation is that the Green  function in $\mathbb R^2$ differs merely from the Green  function in a bounded region by a smooth function, thus the Green  function in $\mathbb R^2$  has the same singularity as that in \eqref{2-15}.

	\section{Simply connected time-periodic solutions}\label{sec3}
	 This section is devoted to the construction of helical Kelvin waves and the proof of Theorem \ref{thm-simply}.

	\subsection{Contour dynamics for helical Kelvin waves}
	Now, we give an equivalent reformulation of \eqref{eq1.6} for uniformly rotating solutions. Let us consider the deformation of a circle with radial $a$, namely
	\begin{equation}\label{eq2.16}
		\sigma(\theta)=R(\theta)e^{i\theta},\ \ \  R(\theta)=\sqrt{a^2+2r(\theta)}e^{i\theta}, \quad {\rm for}\ \theta\in\mathbb{T}.
	\end{equation}
	\begin{proposition}\label{CDE-s}
		Let $\Omega\in \mathbb{R}$. The deformation defined by \eqref{eq2.16} generates a uniformly rotating vortex patch solution with angular velocity $\Omega$ to equations \eqref{eq1.6} if and only if it satisfies the nonlinear equation
		\begin{equation*}
			\mathcal{F}(\Omega,r)=0,
		\end{equation*}
		where, for $\theta\in\mathbb{T}$, 
		\begin{equation}\label{eq2-17}
			\mathcal{F}(\Omega,r)(\theta)=\Omega r'(\theta)+\partial_\theta\left(\int_{0}^{2\pi}\int_{0}^{R(\phi)}G_H\left(R(\theta) e^{i\theta},\rho e^{i\phi}\right)\rho \,d\rho d\phi\right).
		\end{equation}
	\end{proposition}
	\begin{proof}
		Assume that the boundary of domain $D$ can be parameterized by $\sigma(\theta)$ in \eqref{eq2.16}. By results in \cite{hmidi2013boundary}, the fact that the patch $1_{D}(e^{-i\Omega t}\cdot)$ solves \eqref{eq1.6} is equivalent to
		\begin{equation*}
			\nabla^{\perp}(\varphi( x)+\frac{\Omega}{2}| x|^2) \cdot \bm{n}(x)=0, \ \  x\in\partial D,
		\end{equation*}
		where $ \bm{n}(x)$ stands for the normal vector of $\partial D$ at $x$. From the above identity, we deduce
		\begin{equation*}
			\varphi( x)+\frac{\Omega}{2}| x|^2={\rm Constant} \quad \text{for}\  x\in\partial D.
		\end{equation*}

		Denoting $x=\sigma(\theta)$ and differentiating the above equation  with respect to $\theta$, we get  \eqref{eq2-17}.
	\end{proof}
	
	 Corollary \ref{trivialsolution} suggests that the above contour dynamics equation has a family of trivial solutions by taking $r=0$. That is, $$\mathcal F(\Om, 0)\equiv 0, \ \ \ \forall\ \Om\in \mathbb R.$$ This fact can also be verified through  direct computations by using the Fourier expansion of $G_H$ obtained in section \ref{sec2}. In whats follows, we will find non-radial solutions bifurcating from this family of trivial solutions.

	\subsection{Functional setting and regularity analysis}
	\subsubsection{Function spaces}
	The objective of this subsection is to introduce the function space necessary for applying the Crandall-Rabinowitz bifurcation theorem \ref{CR}. We will address our functional within the framework of Hölder spaces. 
	
	More precisely, for a fixed integer $m\in \mathbb{N}$, the real Banach space $\mathbb{X}_m$ and $\mathbb{Y}_m$ are defined by 
	\begin{equation*}
		\mathbb{X}_{m}:=\left\{f\in C^{1,1/2}(\mathbb{T}):~f(\theta)=\sum_{n=1}^{\infty}f_n\cos(nm\theta),~f_n\in\mathbb{R},~\theta\in\mathbb{T}\right\}
	\end{equation*}
	and
	\begin{equation*}
		\mathbb{Y}_m:=\left\{g\in C^{1/2}(\mathbb{T}):~g(\theta)=\sum_{n=1}^{\infty} g_n\sin(nm\theta),~g_n\in\mathbb{R},~\theta\in\mathbb{T}\right\}
	\end{equation*}
	and equipped with the standard $C^{1,1/2}$ and $C^{1/2}$ norms, respectively. On the other hand, we will not define $\mathcal{F}$ on the entire $\mathbb{X}_m$ but a small open subset of $\mathbb{X}_m$, namely
	\begin{equation*}
		\mathbb{B}_{m,\varepsilon}:=\left\{f\in\mathbb{X}_m:\|f\|_{C^{1,1/2}}<\varepsilon\right\},
	\end{equation*}
	for some small $\varepsilon>0$. Later, we will consider $\mathcal{F}$ as a mapping from $\mathbb{R}\times \mathbb{B}_{m,\varepsilon}$ to $\mathbb{Y}_m$, which is well-defined and sufficiently smooth when $\varepsilon$ is sufficiently small.

	For the convenience of later calculations, we define some basic operations. Given two vectors $ x=(x_1,x_2)$ and $ y=(y_1,y_2)$, we define the tensor-product
	\begin{equation*}
		x\otimes y:=\begin{pmatrix}
			x_1y_1&x_1y_2\\
			x_2y_1&x_2y_2
		\end{pmatrix}.
	\end{equation*}
	For any given two matrices $A=\left(A_{ij}\right)_{1\leq i,j\leq 2}$ and $B=\left(B_{ij}\right)_{1\leq i,j\leq 2}$, we define the double contraction of $A$ and $B$ by
	\begin{equation*}
		A:B:=\sum_{1\leq i,j\leq 2}A_{ij}B_{ij}.
	\end{equation*}
	A simple relationship between the tensor-product and the double contraction is given by
	\begin{equation*}
	\left(	 x\otimes y\right):\left(\textbf{a}\otimes\textbf{b}\right)=\left( x\cdot\textbf{a}\right)\left( y\cdot\textbf{b}\right),
	\end{equation*}
	where $ x$,  $ y$,  $\textbf{a}$ and  $\textbf{b}$ are vectors in $\mathbb{R}^2$.

\subsubsection{Regularity of the functional $\mathcal{F}$}
In this subsection, we leverage the property that the Green function in the entire space and the Green function in a bounded domain exhibit the same singularity locally. We will replace the Green function in the functional $\mathcal{F}$ (from the entire space) with that in the bounded domain to verify regularity. More precisely, we have the following proposition.
\begin{proposition}\label{prop3-2}
	Given $m\in\mathbb{N}$ and $\varepsilon>0$ small enough, the functional
	\begin{equation*}
		\mathcal{F}:\mathbb{R}\times\mathbb{B}_{m,\varepsilon}\to\mathbb{Y}_m
	\end{equation*}
	is well-defined and belongs to $C^1\left(\mathbb{R}\times\mathbb{B}_{m,\varepsilon},\mathbb{Y}_m\right)$. 
	Moreover, its mixed partial derivative $\partial_\Omega d_r\mathcal{F}$ exists in the sense that
	\begin{equation*}
		\partial_\Omega d_r\mathcal{F}:\mathbb{R}\times\mathbb{B}_{m,\varepsilon}\to\mathcal{L}(\mathbb{X}_m,\mathbb{Y}_m)
	\end{equation*}
	and belongs to $C^0\left(\mathbb{R}\times\mathbb{B}_{m,\varepsilon},\mathcal{L}(\mathbb{X}_m,\mathbb{Y}_m)\right)$.
\end{proposition}
\begin{proof}
	By the definition of $\mathcal{F}$ in \eqref{eq2-17} and the properties of $G_H$, it suffices to show that the mapping
	\begin{equation*}
		r\mapsto\partial_\theta\left(\int_D G_H(R(\theta)e^{i\theta}, y)d y\right)
	\end{equation*}
	belongs to $C^1\left(\mathbb{B}_{m,\varepsilon},\mathbb{Y}_m\right)$.

	Assuming that $\varepsilon<1$, we can then find a large ball $B=B_{\hat{\rho_0}}(0)\subset\mathbb{R}^2$ such that the bounded connected component $D$ of  $\mathbb{R}^2\setminus\partial D:=\mathbb{R}^2\setminus\left\{(R(\theta)=\sqrt{a^2+2r(\theta)},\theta):\theta\in\mathbb{T}\right\}$ is contained compactly in $B$ for any $r\in \mathbb{B}_{m,\varepsilon}$. Since
	\begin{equation*}
		-\divv \left(K_H\nabla \left(G_H(\cdot,y)-G_{H,B}(\cdot,y)\right)\right)=0~~~\text{in} ~~B,
	\end{equation*}
	by elliptic regularity, we have that the function $H(x,y):=G_H(x,y)-G_{H,B}(x,y)$ belongs to $C^\infty\left(\overline{D\times D}\right)$ for any $r\in \mathbb{B}_{m,\varepsilon}$ (note that $D$ depends on $r$). Then we can rewrite
	\begin{equation*}
		\partial_\theta\left(\int_DG_H(R(\theta)e^{i\theta}, y)d y\right)=\partial_\theta\left(\int_DH(R(\theta)e^{i\theta}, y)d y\right)+\partial_\theta\left(\int_DG_{H,B}(R(\theta)e^{i\theta}, y)d y\right),
	\end{equation*}
	where the first term of the right hand side of above equation has obviously $C^1$ regularity with respect to $r$ due to the regularity of $H$. Hence the problem is reduced to study the regularity of the second term. Furthermore, combining the expansion \eqref{2-15} of $G_{H,B}$, Remark \ref{re2-4} and elliptic regularity as above, we only need to studying the most singular term, namely
	\begin{equation*}
		r\mapsto \mathcal{S}(r)(\theta)=\partial_\theta\left(\int_DA(y)\ln\frac{1}{\left|T_y(R(\theta)e^{i\theta}- y)\right|}d y\right),
	\end{equation*}
	where $A(y)=\sqrt{\det K(y)}^{-1}$ . Direct computation gives
	\begin{align*}
		\mathcal{S}(r)(\theta)=&~\partial_\theta\left(R(\theta)e^{i\theta}\right)\cdot\int_D A( y)\nabla_ x\ln\frac{1}{|T_y( x- y)|}\Big|_{ x=R(\theta)e^{i\theta}}d y\\
		=&~-\partial_\theta\left(R(\theta)e^{i\theta}\right)\cdot\int_D A( y)\nabla_ y\ln\frac{1}{|T_ y( x- y)|}\Big|_{ x=R(\theta)e^{i\theta}}d y\\
		&~+\partial_\theta\left(R(\theta)e^{i\theta}\right)\cdot\int_DA( y)\frac{Q( x, y)T_ y( x- y)}{|T_ y( x- y)|^2}\Big|_{ x=R(\theta)e^{i\theta}}d y\\
		=:&\mathcal{S}_1(r)(\theta)+\mathcal{S}_2(r)(\theta),
	\end{align*}
	where
	\begin{equation*}
		Q( x, y)=
		\begin{pmatrix}
			(x_1-y_1)\partial_{y_1}T_ y^{11}&(x_1-y_1)\partial_{y_1}T_ y^{21}\\
			(x_2-y_2)\partial_{y_2}T_ y^{12}&(x_2-y_2)\partial_{y_2}T_ y^{22}
		\end{pmatrix}.
	\end{equation*}
	One can easily check that
	\begin{equation}\label{eq2-8}
		\frac{A( y)Q( x, y)T_ y( x- y)}{|T_ y( x- y)|^2}=O(1),~~~~\nabla_ x\left(\frac{A( y)Q( x, y)T_ y( x- y)}{|T_ y( x- y)|^2}\right)=O\left(\frac{1}{| x- y|}\right),
	\end{equation}
	as $ x\to y$. Using this fact, we can derive that $\mathcal{S}_2\in C^1\left(B_{m,\varepsilon},\mathbb{Y}_m\right)$. Now, we are going to study the more singular term $\mathcal{S}_1$. Note that the derivatives of integrand involved in $\mathcal{S}_2$ are not integrable. We will integrate by parts before differentiating $\mathcal{S}_1$ with respect to $r$. Hence, we rewrite $\mathcal{S}_1$ as follows
	\begin{align*}
		\mathcal{S}_1(r)(\theta)=&\int_{0}^{2\pi}A(R(\phi))\ln\frac{1}{|T_{R(\phi)e^{i\phi}}(R(\theta)e^{i\theta}-R(\phi)e^{i\phi})|}\mathcal{W}(r)(\theta,\phi)d\phi\\
		&+\partial_\theta\left( R(\theta)e^{i\theta}\right)\cdot\int_{0}^{2\pi}\int_{0}^{R(\phi)}A'(\rho)\ln\frac{1}{|T_{\rho e^{i\phi}}(R(\theta)e^{i\theta}-\rho e^{i\phi})|}\rho e^{i\phi}d\rho d\phi\\
		=:&~~\mathcal{S}_{11}(r)(\theta)+\mathcal{S}_{12}(r)(\theta),
	\end{align*}
	where
	\begin{equation*}
		\mathcal{W}(r)(\theta,\phi)=\partial_{\theta\phi}^2\left(R(\theta)R(\phi)\sin(\theta-\phi)\right).
	\end{equation*}
	Differentiating $\mathcal{S}_{11}$ with respect to $r$, we have
	\begin{align*}
		&d_r\mathcal{S}_{11}(r)[f](\theta)\\
		&=-\int_{0}^{2\pi}A(R(\phi))\frac{T'_{R(\phi) e^{i\phi}}T_{R(\phi) e^{i\phi}}(R(\theta)e^{i\theta}-R(\phi) e^{i\phi})}{|T_{R(\phi) e^{i\phi}}(R(\theta)e^{i\theta}-R(\phi) e^{i\phi})|^2}\left(\frac{f(\theta)}{R(\theta)}-\frac{f(\phi)}{R(\phi)}\right)\mathcal{W}(r)(\theta,\phi)d\phi\\
		&\ \ \ \ +\int_{0}^{2\pi}A(R(\phi))\frac{Q(R(\theta) e^{i\theta},R(\phi) e^{i\phi})T_{R(\phi)e^{i\phi}}(R(\theta)e^{i\theta}-R(\phi) e^{i\phi})}{|T_{R(\phi) e^{i\phi}}(R(\theta)e^{i\theta}-R(\phi) e^{i\phi})|^2}\frac{f(\phi)}{R(\phi)}\mathcal{W}(r)(\theta,\phi)d\phi\\
		&\ \ \ \ +\int_{0}^{2\pi}A(R(\phi))\ln\frac{1}{|T_{R(\phi) e^{i\phi}}(R(\theta)e^{i\theta}-R(\phi) e^{i\phi})|}\partial_{\theta\phi}^2\left(\frac{f(\phi)}{R(\phi)}R(\theta)\sin(\theta-\phi)\right)d\phi\\
		&\ \ \ \ +\partial_\theta\left(\frac{f(\theta)}{R(\theta)}e^{i\theta}\right)\cdot\int_{0}^{2\pi}A(R(\phi))\ln\frac{1}{|T_{R(\phi) e^{i\phi}}(R(\theta)e^{i\theta}-R(\phi) e^{i\phi})|}\partial_\phi\left(iR(\phi)e^{i\phi}\right)d\phi.
	\end{align*}
	Using the Corollary \ref{co1} and Corollary \ref{co2}, we deduce that $d_r\mathcal{S}_{11}[h]\in C^{1/2}(\mathbb{T)}$.
	As for the term $\mathcal{S}_{12}$, direct computation gives
	\begin{align*}
		&d_r\mathcal{S}_{12}(r)[f](\theta)\\
		&=\int_{0}^{2\pi}A'(R(\phi))\ln\frac{1}{|T_{R(\phi) e^{i\phi}}(R(\theta)e^{i\theta}-R(\phi) e^{i\phi})|}e^{i\phi}\cdot\partial_\theta\left(R(\theta)e^{i\theta}\right)f(\phi)d\phi\\
		&\quad+\int_{0}^{2\pi}\int_{0}^{R(\phi)}A'(\rho)\nabla_ x\ln\frac{1}{|T_{\rho e^{i\phi}}( x-\rho e^{i\phi})|}\Big|_{ x=R(\theta)e^{i\theta}}\frac{f(\theta)}{R(\theta)}e^{i\theta}e^{i\phi}\cdot\partial_\theta\left(R(\theta)e^{i\theta}\right)\rho d\rho d\phi\\
		&\quad+\int_{0}^{2\pi}\int_{0}^{R(\phi)}A'(\rho)\ln\frac{1}{|T_{\rho e^{i\phi}}(R(\theta)e^{i\theta}-\rho e^{i\phi})|}e^{i\phi}\cdot\partial_\theta\left(\frac{f(\theta)}{R(\theta)}e^{i\theta}\right)\rho d\rho d\phi\vspace{2em}\\
		&=:\mathcal{S}_{121}(r)[f](\theta)+\mathcal{S}_{122}(r)[f](\theta)+\mathcal{E}(r)[f](\theta).
	\end{align*}
	By Corollary \ref{co1}, we have $\mathcal{S}_{121}\in C^{1/2}(\mathbb{T})$. For $\mathcal{E}$, we have
	\begin{align*}
		\mathcal{E}(r)[f](\theta)&=\partial_\theta\left(\frac{f(\theta)}{R(\theta)}e^{i\theta}\right)\cdot\int_D\ln\frac{1}{|T_ y(R(\theta)e^{i\theta}- y)|}\nabla_ yA(| y|)d y\\
		&=:\partial_\theta\left(\frac{f(\theta)}{R(\theta)}e^{i\theta}\right)\cdot \mathcal{N}(R(\theta)e^{i\theta}).
	\end{align*}
	Since $\nabla A\in L^\infty(D)$, the function $\mathcal{N}$ belongs to $C^1(\mathbb{R}^2)$, which implies $\mathcal{N}(R(\cdot)e^{i(\cdot)})\in C^{1/2}(\mathbb{T})$. As a result, we show that $\mathcal{E}(r)[f]\in C^{1/2}(\mathbb{T})$.

	For the remaining term $\mathcal{S}_{122}$, we have
	\begin{align*}
		\mathcal{S}_{122}(r)[f](\theta)=&-\partial_\theta(R(\theta)e^{i\theta})\cdot\int_D\frac{f(\theta)e^{i\theta}}{R(\theta)}\cdot\nabla_{ y}\ln\frac{1}{|T_ y(R(\theta)e^{i\theta}- y)|}\nabla_ yA(| y|)d y\\
		&+\partial_\theta(R(\theta)e^{i\theta})\cdot\left(\int_D\frac{e^{i\theta}}{R(\theta)}\cdot\frac{Q(R(\theta)e^{i\theta}, y)T_ y(R(\theta)e^{i\theta}- y)}{|T_ y(R(\theta)e^{i\theta}- y)|^2}\nabla_ yA(| y|)d y\right)f(\theta)\\
		=&:\mathcal{S}_{1221}(r)[f](\theta)+\mathcal{S}_{1222}(r)[f](\theta).
	\end{align*}
	By \eqref{eq2-8}, we deduce that $\mathcal{S}_{1222}(r)[f]\in C^{1/2}(\mathbb{T})$. For the other term $\mathcal{S}_{1221}$, integrating by parts, we get
	\begin{align*}
		&\mathcal{S}_{1221}(r)[f](\theta)\\
		&=\left(\partial_\theta(R(\theta)e^{i\theta})\otimes\frac{f(\theta)e^{i\theta}}{R(\theta)}\right):\int_D\ln\frac{1}{|T_ y(R(\theta)e^{i\theta}- y)|}\nabla_ y^2A(| y|)d y\\
		&\quad+\left(\partial_\theta(R(\theta)e^{i\theta})\otimes\frac{f(\theta)e^{i\theta}}{R(\theta)}\right):\int_{0}^{2\pi}\ln\frac{1}{|T_{R(\phi)e^{i\phi}}(R(\theta)e^{i\theta})-R(\phi)e^{\phi})|}\nabla A(R(\phi))\otimes(iR(\phi)e^{i\phi})d\phi.\\
		&=:\mathcal{G}(r)[f](\theta)+\mathcal{I}(r)[f](\theta).
	\end{align*}
	Due to the fact that $\nabla^2 A\in L^\infty(D)$ and using Corollary \ref{co1}, we get $\mathcal{S}_{1221}(r)[f]\in C^{1/2}(\mathbb{T})$. Combining all the computations above, we obtain that $d_r\mathcal{S}(r)[f]\in C^{1/2}(\mathbb{T})$ for all $f\in\mathbb{X}_m$ and $r\in\mathbb{B}_{m,\varepsilon}$ with small $\varepsilon>0$.

	To complete the proof, we need to check that $d_r\mathcal{S}\in C^0(\mathbb{B}_{m,\varepsilon},\mathcal{L}(\mathbb{X}_m,\mathbb{Y}_m))$, i.e.
	\begin{equation*}
		\lim_{\|r_1-r_2\|_{C^{1,1/2}}\to 0}\sup_{\|f\|_{C^{1,1/2}}<1}\|d_r\mathcal{S}(r_1)[f]-d_r\mathcal{S}(r_2)[f]\|_{C^{1/2}}= 0
	\end{equation*}
	To estimate $d_r\mathcal{S}(r_1)[f]-\mathcal{S}(r_2)[f]$, we only need to deal with $\mathcal{E}(r_1)[f]-\mathcal{E}(r_2)[f]$ and $\mathcal{G}(r_1)[f]-\mathcal{G}(r_2)[f]$, since the estimates of the other terms can be concluded directly by Corollary \ref{co1} and Corollary \ref{co2}. Next, we only study $\mathcal{E}(r_1)[f]-\mathcal{E}(r_2)[f]$ since the proof for  $\mathcal{G}(r_1)[f]-\mathcal{G}(r_2)[f]$ is the same. In the following, we write $R(r_1)$ and $R(r_2)$ to emphasize the dependence of $R$ on $r$. 

	We write
	\begin{align*}
		\mathcal{E}(r_1)[f](\theta)-\mathcal{E}(r_2)[f](\theta)=&\partial_\theta\left(\frac{f(\theta)}{R(r_1)(\theta)}e^{i\theta}-\frac{f(\theta)}{R(r_2)(\theta)}e^{i\theta}\right)\cdot\mathcal{N}(R(r_1)e^{i\theta})\\
		&+\partial_\theta\left(\frac{f(\theta)}{R(r_2)}e^{i\theta}\right)\cdot\left(\mathcal{N}(R(r_1)e^{i\theta})-\mathcal{N}(R(r_2)e^{i\theta})\right)\\
		=&:\mathcal{I}[f](\theta)\cdot\mathcal{N}(R(r_1)e^{i\theta})+\partial_\theta\left(\frac{f(\theta)}{R(r_2)}e^{i\theta}\right)\cdot\mathcal{J}(\theta).
	\end{align*}
	Direct computation shows that 
	\begin{equation*}
		\|\mathcal{I}[f](\cdot)\|_{C^{1/2}}\lesssim \|f\|_{C^{1,1/2}}\|r_1-r_2\|_{C^{1,1/2}}.
	\end{equation*}
	For the term $\mathcal{J}$, we write
	\begin{align*}
		\mathcal{J}(\theta)=&\int_{D(r_1)}\left(\ln\frac{1}{|T_ y(R(r_1)(\theta)e^{i\theta}- y)|}-\ln\frac{1}{|T_ y(R(r_2)(\theta)e^{i\theta}- y)|}\right)\nabla_ yA(| y|)d y\\
		&+\left(\int_{D(r_1)}-\int_{D(r_2)}\right)\ln\frac{1}{|T_ y(R(r_2)(\theta)e^{i\theta}- y)|}\nabla_ yA(| y|)d y\\
		=&:\mathcal{J}_1(\theta)+\mathcal{J}_2(\theta).
	\end{align*}
	Denoting
	\begin{equation*}
		x_s(\theta):=sR(r_1)(\theta)+(1-s)R(r_2)(\theta),
	\end{equation*}
	then we have
	\begin{align*}
		\mathcal{J}_1(\theta)=&(R(r_1)(\theta)-R(r_2)(\theta))e^{i\theta}\cdot\int_{0}^{1}\nabla_ x\int_{D(r_1)}\ln\frac{1}{|T_ y( x- y)|}\nabla_ yA(|y|)d y\Big|_{x=x_s(\theta)e^{i\theta}}ds\\
		=&:(R(r_1)(\theta)-R(r_2)(\theta))e^{i\theta}\cdot\int_{0}^{1}\nabla E(x_s(\theta)e^{i\theta})ds,
	\end{align*}
	which implies that 
	\begin{equation*}
		\|\mathcal{J}_1(\theta)\|_{C^{1/2}}\lesssim\left(1+\|R(r_1)\|_{C^1}+\|R(r_2)\|_{C^1}\right)\|\nabla E\|_{C^{1/2}(B)}\|r_1-r_2\|_{C^{1,1/2}}.
	\end{equation*}
	In the following, we are going to estimate the term $\|\nabla E\|_{C^{1/2}}$. For any $x_1,x_2\in B$ and $y\in D(r_1)$, we have
	\begin{equation*}
		\left|\frac{x_1-y}{|T_y(x_1-y)|}-\frac{x_2-y}{|T_y(x_2-y)|}\right|\lesssim_{c_1,c_2,\hat{\rho_0}}\frac{|x_1-x_2|}{\min\left\{|x_1-y|,|x_2-y|\right\}^2}
	\end{equation*}
and
\begin{equation*}
		\left|\frac{x_1-y}{|T_y(x_1-y)|}-\frac{x_2-y}{|T_y(x_2-y)|}\right|\lesssim_{c_1,\hat{\rho_0}}\frac{1}{\min\left\{|x_1-y|,|x_2-y|\right\}},
\end{equation*}
where $c_1:=\inf_{y\in B}\|T_y\|$,  $c_2:=\sup_{y\in B}\|T_y\|$ and $\hat{\rho_0}$ is the radius of the ball $B$. Combining these bounds, we get
\begin{align*}
		\left|\frac{x_1-y}{|T_y(x_1-y)|}-\frac{x_2-y}{|T_y(x_2-y)|}\right|&=\left|\frac{x_1-y}{|T_y(x_1-y)|}-\frac{x_2-y}{|T_y(x_2-y)|}\right|^{1/2}\left|\frac{x_1-y}{|T_y(x_1-y)|}-\frac{x_2-y}{|T_y(x_2-y)|}\right|^{1/2}\\
		&\lesssim_{c_1,c_2,\hat{\rho_0}}\frac{|x_1-x_2|^{1/2}}{\min\left\{|x_1-y|,|x_2-y|\right\}^{1+1/2}}.
\end{align*}
Furthermore, we have
\begin{align*}
	\left|\nabla E(x_1)-\nabla E(x_2)\right|&\lesssim_{c_2}\int_{\mathbb{R}^2}\left|\frac{x_1-y}{|T_y(x_1-y)|}-\frac{x_2-y}{|T_y(x_2-y)|}\right||\nabla A(|y|)|1_{D}(y)dy\\
	&\lesssim_{c_1,c_2,\hat{\rho_0}}|x_1-x_2|^{1/2}\int_{\mathbb{R}^2}\frac{|\nabla A(|y|)|1_{D}(y)}{\min\left\{|x_1-y|,|x_2-y||\right\}^{1+1/2}}dy\\
	&\lesssim_{c_1,c_2,\hat{\rho_0}}|x_1-x_2|^{1/2}\left(\int_B\frac{\nabla_yA(|y|)}{|x_1-y|^{1+1/2}}dy+\int_B\frac{\nabla_yA(|y|)}{|x_2-y|^{1+1/2}}dy\right)\\
	&\lesssim_{c_1,c_2,\hat{\rho_0}}|x_1-x_2|^{1/2}.
\end{align*}
	Hence $\|\nabla E\|_{C^{1/2}(B)}$ can be controlled by a constant independent of $r_1$. Furthermore, we get
	\begin{equation*}
		\|\mathcal{J}_1(\theta)\|_{C^{1/2}}\lesssim\|r_1-r_2\|_{C^{1,1/2}}.
	\end{equation*}
	As for the term $\mathcal{J}_2$, using change of variables in polar coordinates, we obtain
	\begin{align*}
		\mathcal{J}_2(\theta)=&\int_{0}^{2\pi}\int_{0}^{1}\Big(\ln\frac{1}{|T_{sR(r_1)(\phi)e^{i\phi}}(R(r_2)(\theta)e^{i\theta}-sR(r_1)(\phi)e^{i\phi})|}\\
		&-\ln\frac{1}{|T_{sR(r_1)(\phi)e^{i\phi}}(R(r_2)(\theta)e^{i\theta}-sR(r_2)(\phi)e^{i\phi})|}\Big)P_1(s,\phi)dsd\phi\\
		&+\int_{0}^{2\pi}\int_{0}^{1}\ln\frac{1}{|T_{sR(r_1)(\phi)e^{i\phi}}(R(r_2)(\theta)e^{i\theta}-sR(r_2)(\phi)e^{i\phi})|}(P_1-P_2)(s,\phi)dsd\phi,
	\end{align*}
	where for $k=1,2$, 
	\begin{equation*}
		P_k(s,\phi)=sR^2(r_k)(\phi)\nabla A(sR(r_k)\phi).
	\end{equation*}
	Thus, by the virtue of Corollary \ref{co1} and Corollary \ref{co2}, for $\beta<1/2$, we get
	\begin{equation*}
		\|\mathcal{J}_2\|_{C^{1/2}}\lesssim\|r_1-r_2\|_{C^1}^\beta\|P_2\|_{L^\infty([0,1]\times\mathbb{R}^2)}+\|r_1-r_2\|_{L^\infty(\mathbb{T})}.
	\end{equation*}
	Now, combining all the preceding bounds, we deduce that
	\begin{equation*}
		\lim_{\|r_1-r_2\|_{C^{1,1/2}}\to 0}\sup_{\|f\|_{C^{1,1/2}}<1}\|\mathcal{E}(r_1)[f]-\mathcal{E}(r_2)[f]\|_{C^{1/2}}= 0,
	\end{equation*}
	which completes the proof.
\end{proof}

	\subsection{Kernel of $d_r\mathcal{F}(\Omega,0)$ and Spectral analysis}
	
	First we determine the kernel of $d_r\mathcal{F}(\Omega,0)$. To do this we need to get formula for $d_r\mathcal{F}$.
\begin{lemma}\label{lem2.2}
	For $f\in\mathbb{X}_m$ taking the form
	\begin{equation*}
		f(\theta)=\sum_{n=1}^{\infty}f_n\cos(nm\theta),
	\end{equation*}
	it holds that
	\begin{equation*}
		d_r\mathcal{F}(\Omega,0)[f](\theta)=-\sum_{n=1}^{\infty}nm\left(\Omega-\frac{h^2+a^2}{2h^2}-\frac{a^2}{h^2}I'_{nm}\left(\frac{nma}{h}\right)K'_{nm}\left(\frac{nma}{h}\right)\right)f_n\sin(nm\theta).
	\end{equation*}
\end{lemma}
\begin{proof}
Let $R_{t,f}=\sqrt{a^2+2(r+tf)}$. Then the derivative of $\mathcal{F}$ with respect to $r$ can be computed as follows
\begin{equation}\label{eq2.7}
	d_r\mathcal{F}(\Omega,r)[f](\theta)=\partial_\theta\left(\int_{0}^{2\pi}\frac{d}{dt}\int_{0}^{R_{t,f}(\phi)}G_H(R_{t,f}(\theta)e^{i\theta},\rho e^{i\phi})\rho d\rho\bigg|_{t=0}d\phi\right)+\Omega f'(\theta).
\end{equation}
Direct computation gives 
\begin{equation*}
	\begin{split}
		&\frac{d}{dt}\int_{0}^{2\pi}\int_{0}^{R_{t,f}(\phi)}G_H(R_{t,f}(\theta)e^{i\theta},\rho e^{i\phi})\rho d\rho d\phi\bigg|_{t=0}\\
		&=\int_{0}^{2\pi} G_H(R_{t,f}(\theta)e^{i\theta},R_{t,f}(\phi)e^{i\phi})f(\phi)  d\phi\bigg|_{t=0} \\
		&\quad+\int_{0}^{2\pi}\int_{0}^{R_{t,f}(\phi)}\nabla G_H(R_{t,f}(\theta)e^{i\theta},\rho e^{i\phi})\cdot\frac{f(\theta)e^{i\theta}}{R_{t,f}(\theta)}\rho d\phi\bigg|_{t=0}\\
		&=\int_{0}^{2\pi} G_H(R(\theta)e^{i\theta},R(\phi)e^{i\phi})f(\phi)  d\phi +\int_{0}^{2\pi}\int_{0}^{R(\phi)}\nabla G_H(R(\theta)e^{i\theta},\rho e^{i\phi})\cdot\frac{f(\theta)e^{i\theta}}{R(\theta)}\rho d\phi.
	\end{split}
\end{equation*}
If we define
\begin{equation*}
	Q(r)(\theta)=\frac{e^{i\theta}}{R(\theta)}\cdot\int_{0}^{2\pi}\int_{0}^{R(\phi)}\nabla G_H(R(\theta)e^{i\theta},\rho e^{i\phi})\rho \,d\rho\,d\phi
\end{equation*}
and
\begin{equation*}
	P(r)[f](\theta)=\int_{0}^{2\pi}G_H(R(\theta)e^{i\theta},R(\phi)e^{i\phi})f(\phi)d\phi,
\end{equation*}
then we have
\begin{equation}\label{eq2.8}
	d_r\mathcal{F}(\Omega,r)[f](\theta)=\partial_\theta\left(Q(r)(\theta)f(\theta)+P(r)[f](\theta)\right)+\Omega f'(\theta).
\end{equation}
For $f(\theta)=\sum_{n=1}^{\infty}f_n\cos(nm\theta)$, we compute the term $P(0)[f]$ firstly, and get
\begin{equation}\label{eq2.9}
	\begin{split}
	P(0)[f](\theta)&=\sum_{n=1}^{\infty}f_n\int_{0}^{2\pi}G_H(ae^{i\theta},ae^{i\phi})\cos(nm\phi)d\phi\\
	&=\sum_{n=1}^{\infty}f_n\int_{0}^{2\pi}G_H(a,ae^{i(\phi-\theta)})\cos(nm\phi)d\phi\\
	&=\sum_{n=1}^{\infty}f_n\int_{0}^{2\pi}G_H(a,ae^{i\phi})\cos(nm(\phi+\theta))d\phi\\
	&=\sum_{n=1}^{\infty}\left(f_n\int_{0}^{2\pi}G_H(a,ae^{i\phi})\cos(nm\phi)d\phi \right)\cos(nm\theta).
	\end{split}
\end{equation}
Using the Green expansion developed in section \ref{sec2}, the Fourier coefficient above can be expressed in more precise form, namely
\begin{equation}\label{eq2.10}
	\begin{split}
		\int_{0}^{2\pi}G_H(a,ae^{i\phi})\cos(nm\phi)d\phi&=-\sum_{k\in\mathbb{Z}}\frac{a^2I'_{k}\left(\frac{ka}{h}\right)K'_{k}\left(\frac{ka}{h}\right)}{2\pi h^2}\int_{0}^{2\pi}\cos(nm\phi)e^{-ik\phi}d\phi\\
		&=-\frac{a^2I'_{nm}\left(\frac{nma}{h}\right)K'_{nm}\left(\frac{nma}{h}\right)}{\pi h^2}\int_{0}^{2\pi}\cos^2(nm\phi)d\phi\\
		&=-\frac{a^2I'_{nm}\left(\frac{nma}{h}\right)K'_{nm}\left(\frac{nma}{h}\right)}{ h^2},
	\end{split}
\end{equation}
and consequently \eqref{eq2.9} and \eqref{eq2.10} lead to
\begin{equation}\label{2.11}
	P(0)[f](\theta)=-\sum_{n=1}^{\infty}\frac{a^2I'_{nm}\left(\frac{nma}{h}\right)K'_{nm}\left(\frac{nma}{h}\right)}{ h^2}f_n\cos(nm\theta).
\end{equation}

Next, we compute the other term $Q(0)(\theta)$. Direct computation shows that
\begin{equation}\label{eq2.12}
	\begin{split}
	Q(0)(\theta)&=\frac{1}{a}\int_{0}^{2\pi}\int_{0}^{a}\nabla_xG_H(ae^{i\theta},\rho e^{i\phi})\cdot e^{i\theta}\rho d\rho d\phi\\
	&=\frac{1}{a}\partial_{x_1}\left(\int_{\mathbb{R}^2}G_H(x,y)1_{B_a(0)}(y)dy\right)\bigg|_{x=(a,0)}\\
	&=:\frac{1}{a}\partial_{x_1}\Psi_1(a,0).
	\end{split}
\end{equation}
It is obvious that $\Psi_1$ is radial and satisfies the following ordinary differential equation
\begin{equation*}
	-\frac{1}{\rho}\partial_\rho\left(\frac{\rho}{1+\rho^2/h^2}\partial_\rho\Psi_1\right)= 1_{\{\rho<a\}}
\end{equation*}
which can be solved by
\begin{equation*}
	\Psi_1(\rho)=-\int_{0}^{\rho}\frac{h^2+s^2}{h^2s}\int_{0}^{\min\left\{s,a\right\}}\tau d\tau ds+\text{constant},
\end{equation*}
and gives
\begin{equation}\label{eq2.13}
	\partial_{x_1}\Psi_1(a,0)=-\frac{h^2+a^2}{2h^2}a.
\end{equation}
Combining \eqref{eq2.12} and \eqref{eq2.13}, we get
\begin{equation}\label{eq2.14}
	Q(0)(\theta)=-\frac{a^2+h^2}{2h^2}.
\end{equation}
Hence, \eqref{eq2.8}, \eqref{2.11} and \eqref{eq2.14} give the conclusion in this Lemma, the proof is thus completed.
\end{proof}

The next step is to investigate the monotonicity of the eigenvalues to ensure that the kernel space is one-dimensional. Essentially, we only need to study the monotonicity of the mapping $\nu \mapsto I_\nu'(\nu z) K_\nu'(\nu z)$. 
\begin{proposition}\label{prop3-6}
	There exist positive numbers $A_1$, $A_2$ and $M(z)$, such that $I_\nu'(\nu z)K_\nu'(\nu z)$ is strictly monotone increasing function of $\nu$, if one of the following three cases holds:
	\begin{itemize}
		\item[(1)] $\nu\geq 2$ and $0<z<A_1$,
		\item[(2)] $\nu\geq 2$ and $z>A_2$,
		\item[(3)] $\nu>M(z)$ for any fixed $z>0$.
	\end{itemize}
\end{proposition}
\begin{proof}
Recall the notations in Lemma \ref{lemma3.4}. We can rewrite $I'_\nu(z)K_\nu'(z)$ as follows
\begin{equation}
	I'_\nu(z)K_\nu'(z)=-\frac{1}{z^2}I_\nu(z)K_\nu(z)C(I_\nu(z))C(K_\nu(z)).
\end{equation}
Then according to the inequalities in Lemma \ref{lemma3.4} and Lemma \ref{lemma3.5}, it follows that
\begin{equation}\label{Bessel1}
	\begin{split}
		&\frac{1}{\nu^2 z^2}\frac{\left(\sqrt{(\nu+\frac{1}{2})^2+\nu^2 z^2}-\frac{1}{2}\right)\left(\sqrt{(\nu-1)^2+\nu^2 z^2}+1\right)}{2\sqrt{(\nu-\frac{1}{2})^2+\nu^2 z^2}}\\
		\geq& -I_\nu'(\nu z)K_\nu'(\nu z)\\
		\geq& \frac{1}{\nu^2 z^2}\frac{\left(\sqrt{(\nu-\frac{1}{2})^2+z^2\nu^2}+\frac{1}{2}\right)\left(\sqrt{(\nu+1)^2+z^2\nu^2}-1\right)}{1+\sqrt{\nu^2+z^2\nu^2}+\sqrt{(\nu-1)^2+z^2\nu^2}}
	\end{split}
\end{equation}  
and  
\begin{equation}\label{Bessel2}
	\begin{split}  
		&\frac{1}{(\nu-1)^2z^2}\frac{\left(\sqrt{(\nu-\frac{1}{2})^2+(\nu-1)^2z^2}-\frac{1}{2}\right)\left(\sqrt{(\nu-2)^2+(\nu-1)^2z^2}+1\right)}{2\sqrt{(\nu-\frac{3}{2})^2+(\nu-1)^2z^2}}\\      
		\geq& -I_{\nu-1}'((\nu-1)z)K_{\nu-1}'((\nu-1)z)\\
		\geq& \frac{1}{(\nu-1)^2z^2}\frac{\left(\sqrt{\nu^2+(\nu-1)^2z^2}-1\right)\left(\sqrt{(\nu-\frac{3}{2})^2+(\nu-1)^2z^2}+\frac{1}{2}\right)}{1+\sqrt{(\nu-1)^2+(\nu-1)^2z^2}+\sqrt{(\nu-2)^2+(\nu-1)^2z^2}}.
	\end{split}
\end{equation}
 
We claim that $I_\nu'(\nu z)K_\nu'(\nu z)$ is a monotonically increasing function of $\nu$ for fixed $z$. In fact, let us first define 
\begin{equation}\label{f}
	\begin{split}
		f_z(\nu)=&2\left(\sqrt{(\nu-\frac{1}{2})^2+\nu^2 z^2}\right)\left(\sqrt{\nu^2+(\nu-1)^2z^2}-1\right)\left(\sqrt{(\nu-\frac{3}{2})^2+(\nu-1)^2z^2}+\frac{1}{2}\right)\\
		&-\frac{(\nu-1)^2}{\nu^2}\left(1+\sqrt{(\nu-1)^2+(\nu-1)^2z^2}+\sqrt{(\nu-2)^2+(\nu-1)^2z^2}\right)\\
		&\quad\times\left(\sqrt{(\nu+\frac{1}{2})^2+\nu^2 z^2}-\frac{1}{2}\right)\left(\sqrt{(\nu-1)^2+\nu^2 z^2}+1\right).
	\end{split}
\end{equation}
Then in order to prove our claim,  it is sufficient to show that $f_z(\nu)>0$  due to the inequalities \eqref{Bessel1} and \eqref{Bessel2}. We will demonstrate that the claim holds true for the following three cases.
\begin{itemize}
	\item   [\textbf{Case 1}.] For $z>0$ sufficiently small and any $\nu\geq 3$.
\end{itemize}
Denote $s=\frac{1}{\nu}\in(0,\frac{1}{3}]$ and define
\begin{equation*}
	\begin{split}
		h_z(s):=&\frac{f_z(\nu)}{\nu^3}\\
		=&2\sqrt{(1-\frac{1}{2}s)^2+z^2}\left(\sqrt{1+(1-s)^2z^2}-s\right)\left(\sqrt{(1-\frac{3}{2}s)^2+(1-s)^2z^2}+\frac{s}{2}\right)\\
		&-(1-s)^2\left(s+(1-s)\sqrt{1+z^2}+\sqrt{(1-2s)^2+(1-s)^2z^2}\right)\\
		&\quad\times\left(\sqrt{(1+\frac{s}{2})^2+z^2}-\frac{s}{2}\right)\left(\sqrt{(1-s)^2+z^2}+s\right).
	\end{split}
\end{equation*}
By direct computation, we have
\begin{gather*}
	\lim_{s\to0^+}h_z(s)=0,\\
	\lim_{s\to0^+}\frac{dh_z(s)}{ds}=2\sqrt{1+z^2}\left(z^2-\frac{3}{2}\sqrt{1+z^2}+2\right)\geq \frac{7}{8}>0,\quad {\rm for\ all}\  z\geq 0,\\
	\lim_{z\to0^+}\frac{dh_z(s)}{dz}=0,\quad {\rm for\ all}\  s\in(0,\frac{1}{3}].
\end{gather*}
Therefore, there exists $\delta_1 >0$ such that 
\begin{equation*}
	h_z(s)>0\quad {\rm for\ any}\ (s,z)\in (0,\delta_1)\times(0,\delta_1).
\end{equation*}
On the other hand, we obtain
\begin{equation*}
	\begin{split}
		\lim_{z\to 0^+}h_z(s)=2\left(1-\frac{s}{2}\right)\left(1-s\right)^2-2(1-s)^3=s(1-s)^2,
	\end{split}
\end{equation*}
which implies that there exists $\delta_2 >0$, such that 
\begin{equation*}
	h_z(s)>\frac{4}{9}\delta_1>0 \quad {\rm for\ any}\ (s,z)\in [\delta_1,\frac{1}{3}]\times(0,\delta_2).
\end{equation*}
Then by taking $A_1=\min\{\delta_1,\delta_2\}$, we obtain that $I'_{\nu}(\nu z)K'_{\nu}(\nu z)$ is strictly monotone increasing about $\nu\geq 2$ for all $z\in(0,A_1)$.

\begin{itemize}
	\item  [\textbf{Case 2}.] For $z$ sufficiently large and any $\nu\geq 3$. 
\end{itemize}
Similarly, we denote
\begin{equation*}
	\begin{split}
		g_z(s):=&\frac{f_z(\nu)}{z^3\nu^3}\\
		=&2\sqrt{\frac{1}{z^2}(1-\frac{1}{2}s)^2+1}\left(\sqrt{\frac{1}{z^2}+(1-s)^2}-\frac{s}{z}\right)\left(\sqrt{\frac{1}{z^2}(1-\frac{3}{2}s)^2+(1-s)^2}+\frac{s}{2z}\right)\\
		&-(1-s)^2\left(\frac{s}{z}+(1-s)\sqrt{1+\frac{1}{z^2}}+\sqrt{\frac{1}{z^2}(1-2s)^2+(1-s)^2}\right)\\
		&\quad\times\left(\sqrt{\frac{1}{z^2}(1+\frac{s}{2})^2+1}-\frac{s}{2z}\right)\left(\sqrt{\frac{1}{z^2}(1-s)^2+1}+\frac{s}{z}\right).
	\end{split}
\end{equation*}
By direct computation, we have
\begin{gather*}
	\lim_{s\to0^+}g_z(s)=0,\\
	\lim_{s\to0^+}\frac{dg_z(s)}{ds}=2\sqrt{1+\frac{1}{z^2}}\left(\frac{2}{z^2}-\frac{3}{2z}\sqrt{\frac{1}{z^2}+1}+1\right)\geq 2\left(\sqrt{2}-\frac{3}{4}\right)^2>0,\ {\rm for\ all}\  z\geq1,\\
	\lim_{z\to+\infty}\frac{dg_z(s)}{dz}=0,\  {\rm for\ all}\  s\in(0,\frac{1}{3}].
\end{gather*}
On the other hand, 
\begin{equation*}
	\begin{split}
		\lim_{z\to +\infty}g_z(s)=2\left(1-s\right)^2-2(1-s)^3=2s(1-s)^2.
	\end{split}
\end{equation*}
Then as in the Case 1, there exists a sufficiently large constant $A_2>0$, such that $I'_{\nu}(\nu z)K'_{\nu}(\nu z)$ is strictly monotone increasing about $\nu\geq 2$ for all $z\in(A_2,+\infty)$.

\begin{itemize}
	\item  [\textbf{Case 3}.] For any fixed $z=z_0>0$ and $\nu$ sufficiently large.
\end{itemize} 
The proof is straightforward. Due to the computation in Case 1, we have 
\begin{equation*}
	\lim_{s\to 0^+}h_{z_0}(s)=0,\quad \lim_{s\to 0^+}\frac{d}{ds}h_{z_0}(s)>0,
\end{equation*}
which implies that there exists $\delta_{z_0}>0$, such that $h_{z_0}(s)>0$ for all $s\in(0,\delta_{z_0})$. Therefore by taking $M(z_0)=[\frac{1}{\delta_{z_0}}]+1$, the proof is completed.
\end{proof}
	\subsection{Proof of Theorem \ref{thm-simply}}\label{sc4}
		With the preceding results and those from the previous section, we are now in a position to establish all the prerequisites for applying Crandall-Rabinowitz Theorem, which will subsequently allow us to deduce Theorem \ref{thm-simply}. This is formalized in the following proposition.
	\begin{proposition}\label{prop3-7}
		Let $A_1$, $A_2$ and $M(z)$ be as in Proposition \ref{prop3-6}. Assume $m>M(a,h):=M(a/h)$ or $h\in(0,H_0(a))\cup(H_1(a),+\infty)$ with $H_0(a):=a/A_2$ and $H_1(a):=a/A_1$, where $M$, $A_1$ and $A_2$ are obtained in Proposition \ref{prop3-6},  then the following assertions hold :
		\begin{itemize}
			\item[(1)]The linearized operator
				\begin{equation*}
					d_r\mathcal{F}(\Omega,0):\mathbb{X}_m\to\mathbb{Y}_m
				\end{equation*}
				has a nontrivial kernel if and only if
				\begin{equation*}
					\Omega\in\left\{\Omega_{nm}:=\frac{a^2+h^2}{2h^2}+\frac{a^2}{2h^2}I'_{nm}\left(\frac{nma}{h}\right)K'_{nm}\left(\frac{nma}{h}\right),~n\in\mathbb{N}\right\}.
				\end{equation*}
			\item[(2)]Moreover,   $\ker\left(d_r \mathcal{F}(\Omega_{ m},0)\right)\subset \mathbb{X}_m$ is a one-dimensional linear space generated by
				\begin{equation*}
					\theta\to\cos(m\theta).
				\end{equation*}
			\item[(3)]Furthermore, the range of $d_r\mathcal{F}(\Omega_{m},0)$ is closed and of co-dimension one.
			\item[(4)] The transversality condition holds, namely
				\begin{equation*}
					\partial_\Omega d_r\mathcal{F}(\Omega_{m},0)[\cos(m\cdot)]\notin {\rm Range}\left(d_r\mathcal{F}(\Omega_{m},0)\right).
				\end{equation*}
		\end{itemize}
	\end{proposition}
	\begin{proof}
		Letting $f$ be a function in $\mathbb{X}_m$ taking the form
		\begin{equation*}
			f(\theta)=\sum_{n=1}^{\infty}f_n\cos(nm\theta),
		\end{equation*}
		in view of Lemma \ref{lem2.2}, we have,  
		\begin{equation}\label{eq2.15}
			d_r\mathcal{F}(\Omega_{m},0)[f](\theta)=-\sum_{n=2}^{\infty}nm\left(\Omega_{ m}-\Omega_{nm}\right)f_n\sin(nm\theta).
		\end{equation}
 
		(1)~Recall that the eigenvalues $\Omega_{nm}$ are strictly monotonic with respect to $n$ due to Proposition \ref{prop3-6}, we deduce that  $\Omega_{ m}\not=\Omega_{nm}$ for $n\geq 2$ and hence ${\rm ker}(d_r\mathcal{F}(\Omega_{ m},0))$ has dimension 1 and can be generated by $\cos( m\cdot)$.

		(2)~According to \eqref{eq2.15}, it is reasonable to conjecture that the range of $d_r\mathcal{F}(\Omega_{ m},0)$ coincides with 
		\begin{equation*}
			\mathbb{Z}_{m }:=\left\{g\in C^\alpha(\mathbb{T}):g(\theta)=\sum_{n=2}^{\infty}g_n\sin(nm\theta),~g_n\in\mathbb{R}\right\}.
		\end{equation*}
		We now prove this is true. In fact, it is sufficient to prove $\mathbb{Z}_m\subset \text{Range}(d_r\mathcal{F}(\Omega_{ m},0))$. For any $g\in\mathbb{Z}_m$ with the form
		\begin{equation*}
			g(\theta)=\sum_{n=2}^{\infty}g_n\sin(nm\theta),
		\end{equation*}
		we define
		\begin{equation*}
			f_n=-\frac{g_n}{nm\left(\Omega_{ m}-\Omega_{nm}\right)},\quad \text{for}\ n\geq 2.
		\end{equation*}
		Using the monotonicity of the eigenvalues, it is obvious that 
		\begin{equation*}
			\frac{1}{\Omega_{ m}-\Omega_{nm}}=O_n(1).
		\end{equation*}
		Defining
		\begin{equation*}
			f(\theta)=\sum_{n=2}^{\infty}f_n\cos(nm\theta),
		\end{equation*}
		then we formally have $d_r\mathcal{F}(\Om_m,0)[f]=g$. In order to make the construction rigorous, we need to show that $f\in C^{1,1/2}$. Firstly, we estimate the $L^\infty$ norm of $h$, namely
		\begin{equation*}
			\begin{split}
				\|f\|_{\infty} \lesssim \sum_{n=2}^{\infty}\frac{|g_n|}{n} \lesssim  \left(\sum_{n=2}^{\infty}\frac{1}{n^2}\right)^{1/2}||g||_2 <\infty.
			\end{split}
		\end{equation*}
		Next, we need to study the derivative of $f(\theta)$. To check $f'\in C^{1/2}$, we decompose $f'$ as follows
		\begin{equation*}
		\begin{aligned}
				f'(\theta)&=\sum_{n=2}^{\infty}\left(\frac{1}{\Omega_{ m}-\Omega_{nm}}-\frac{1}{\Omega_{ m}}\right)g_n\sin(nm\theta)+\frac{1}{\Omega_{ m}}g(\theta)\\
				&=W*g(\theta)+\frac{1}{\Omega_{ m}}g(\theta),
		\end{aligned}
		\end{equation*}
		where
		\begin{equation*}
			W(\theta)=\sum_{n=2}^{\infty}\left(\frac{1}{\Omega_{ m}-\Omega_{nm}}-\frac{1}{\Omega_{ m}}\right)\sin(nm\theta)
		\end{equation*}
		and the sum above converges in $L^2$, since 
		\begin{equation*}
			\frac{1}{\Omega_{ m}-\Omega_{nm}}-\frac{1}{\Omega_{ m}}=O\left(\left|I'_{nm}\left(\frac{nma}{h}\right)K'_{nm}\left(\frac{nma}{h}\right)\right|\right)=O\left(\frac{1}{n}\right).
		\end{equation*}
		Now, using the fact that $L^1*C^{1/2}\to C^{1/2}$, we conclude $f'\in C^{1/2}$. Hence, it implies $\mathbb{Z}_{m }=\text{Range}(d_r\mathcal{F}(\Omega_{ m},0))$, which is closed and of co-dimension one.

		(3)The transversality condition follows from the fact that
		\begin{equation*}
			\partial_\Omega d_r\mathcal{F}(\Omega_{ m},0)[\cos( m\cdot)](\theta)=- m\sin( m\theta)\notin\mathbb{Z}_{m }.
		\end{equation*}
		The proof is thus completed.
	\end{proof}

	\begin{proof}[\textbf{Proof of Theorem }\ref{thm-simply}]
		\
		
		In view of Corollary \ref{trivialsolution}, Proposition \ref{prop3-2} and Proposition \ref{prop3-6}, the functional $\mathcal{F}$ satisfies all the assumptions required in Theorem \ref{CR}, then the conclusion in Theorem \ref{CR} give the desired results in Theorem \ref{thm-simply}, thus the proof is completed.
	\end{proof}

	\begin{remark}\label{rem}
			It can be seen that the   conditions that either  
			$m>M(a,h)$ or $h\in (0, H_0(a))\cup (H_1(a), +\infty)$   in Theorem \ref{thm-simply}  come from Proposition \ref{prop3-6}  when we verify whether the function $f_z(\nu)$ defined in \eqref{f} is greater than zero. We remark that numerical computations show that $f_z(\nu)>0$ for all $z>0$ and $\nu\geq3$, which indicates that  the conditions that either  
			$m>M(a,h)$ or $h\in (0, H_0(a))\cup (H_1(a), +\infty)$  in Theorem \ref{thm-simply} can be actually removed. However, to achieve this goal,  more detailed estimates and calculations are in demand.
	\end{remark}

\section{Doubly connected time-periodic solutions}\label{sec4}

	In this section, we construct uniformly rotating $m$-fold symmetric vortex patches for \eqref{eq1.6} and prove Theorem \ref{thm-doubly}.

	Let $a_1>a_2>0$ be the two numbers in Theorem \ref{thm-doubly}.  We seek for solutions of the form 
	$$\omega(x,t)=1_{e^{i\Om t}( D_1\setminus \overline{D_2})},$$ 
	with $D_k$ being a perturbation of the disk $B_{a_k}(0)$ for $k=1,2$. We parameterize $\partial D_k$ by $$ \sigma_k: \mathbb T\mapsto \partial D_k$$ with 
	\begin{equation}\label{param-d}
		\sigma_k(\theta)=R_k(\theta)e^{i\theta},\ \ \  R_k(\theta)=\sqrt{a_k+r_k(\theta)}.
	\end{equation}

	By \eqref{eq1.6} and \eqref{param-d}, and an argument similar to the proof of Proposition \ref{CDE-s}, we obtain the equations for $r_k$:
	\begin{equation*}
			\mathcal{F}_k(\Omega,r_1, r_2)=0, \quad  k=1,2,
	\end{equation*}
	where, for $\theta\in\mathbb{T}$, we denote
	\begin{equation}\label{eq4-1}
		\begin{split}
			\mathcal{F}_k(\Omega,r_1, r_2)(\theta)=\Omega r'_k(\theta)+\partial_\theta&\left(\int_{0}^{2\pi}\int_{0}^{R_1(\phi)}G_H\left(R_k(\theta) e^{i\theta},\rho e^{i\phi}\right)\rho\,d\rho\,d\phi\right.\\
			&\quad\left.-\int_{0}^{2\pi}\int_{0}^{R_2(\phi)}G_H\left(R_k(\theta) e^{i\theta},\rho e^{i\phi}\right)\rho\,d\rho\,d\phi\right).
		\end{split}
	\end{equation}
	
	 In view of Corollary \ref{trivialsolution}, by  direct computations,  we can check that the above contour dynamics equation has  a family of trivial solutions: $$\mathcal F_k(\Om, 0, 0)\equiv 0, \ \ \ \forall\ \Om\in \mathbb R,\  \ \ k=1,2.$$ In the following, we will find non-radial solutions bifurcating from this family of trivial solutions.

	In order to apply the Crandall-Rabinowitz theorem, the first step is to verify certain regularity properties of $\mathcal{F}_k$, which we collect in the following lemma. 
	\begin{lemma}\label{lem4-1}
		Given $m\in\mathbb{N}$ and $\varepsilon>0$ small enough, for $k=1,2$, the functional
		\begin{equation*}
			\mathcal{F}_k:\mathbb{R}\times\mathbb{B}_{m,\varepsilon}\times\mathbb{B}_{m,\varepsilon}\to\mathbb{Y}_m\times\mathbb{Y}_m
		\end{equation*}
		is well-defined and belongs to $C^1\left(\mathbb{R}\times\mathbb{B}_{m,\varepsilon}\times\mathbb{B}_{m,\varepsilon},\mathbb{Y}_m\times\mathbb{Y}_m\right)$. 
		Moreover, its mixed partial derivative $\partial_\Omega d_{r_1,r_2}\mathcal{F}_k$ exists in the sense that
		\begin{equation*}
			\partial_\Omega d_{r_1,r_2}\mathcal{F}_k:\mathbb{R}\times\mathbb{B}_{m,\varepsilon}\times\mathbb{B}_{m,\varepsilon}\to\mathcal{L}(\mathbb{X}_m\times\mathbb{X}_m,\mathbb{Y}_m\times\mathbb{Y}_m)
		\end{equation*}
		and belongs to $C^0\left(\mathbb{R}\times\mathbb{B}_{m,\varepsilon}\times\mathbb{B}_{m,\varepsilon},\mathcal{L}(\mathbb{X}_m\times\mathbb{X}_m,\mathbb{Y}_m\times\mathbb{Y}_m)\right)$.
	\end{lemma}
	The above lemma can be proved by adapting argument of the proof for Proposition \ref{prop3-2} without introducing any new ideas. So we omit the details here. 
	\smallskip
	
	Next, we shall linearize $\mathcal{F}_k$ at the trivial solution $(\Om, 0, 0)$. 
	\begin{lemma}\label{Linearization-d}
		Denote $\bm{\mathcal{F}}=(\mathcal{F}_1, \mathcal{F}_2)$. For $f_1, f_2\in\mathbb{X}_m$ taking the form
		\begin{equation*}
			f_1(\theta)=\sum_{n=1}^{\infty}f_n^{(1)}\cos(nm\theta), \ \ \ f_2(\theta)=\sum_{n=1}^{\infty}f_n^{(2)}\cos(nm\theta),
		\end{equation*}
		it holds that
		\begin{equation*}
			d_{r_1,r_2}\bm{\mathcal{F}}(\Omega,0,0)[f_1,f_2](\theta)=-\sum_{n=1}^{\infty}nm M_{nm}(\Om, a_1, a_2) \begin{pmatrix}
				f_{n}^{(1)} \\
				f_n^{(2)}
			\end{pmatrix}\sin(nm\theta),
		\end{equation*} where the matrix $	M_n(\Om, a_1, a_2)$ is given by 
		\begin{equation*}
			M_n(\Om,a_1,a_2)=\begin{pmatrix}
				\Om-\frac{a_1^2}{h^2}I_n'\left(\frac{na_1}{h}\right)K_n'\left(\frac{na_1}{h}\right)-\frac{(h^2+a_1^2)(a_1^2-a_2^2)}{2h^2a_1^2}&-\frac{a_1a_2}{h^2}I_n'\left(\frac{na_2}{h}\right)K_n'\left(\frac{na_1}{h}\right)\\
				\frac{a_1a_2}{h^2}I_n'\left(\frac{na_2}{h}\right)K_n'\left(\frac{na_1}{h}\right)&\Om+\frac{a_2^2}{h^2}I_n'\left(\frac{na_2}{h}\right)K_n'\left(\frac{na_2}{h}\right)
			\end{pmatrix}.
		\end{equation*}
		
	\end{lemma}
	\begin{proof}
		We first compute $d_{r_1}\mathcal{F}_1(\Om, 0, 0)[f_1]$. By comparing the expressions of $\mathcal{F}_1$ in \eqref{eq4-1} and $\mathcal{F}$ in \eqref{eq2-17} , and using the results from Lemma \ref{lem2.2}, it is not difficult to see that only the derivative of $\partial_\theta \int_{0}^{2\pi}\int_{0}^{R_2(\phi)}G_H\left(R_1(\theta) e^{i\theta},\rho e^{i\phi}\right)\rho d\rho d\phi$ with respect to $r_1$ is needed to be determined. Denote $$ F_{1,3}(r_1, r_2)= \int_{0}^{2\pi}\int_{0}^{R_2(\phi)}G_H\left(R_1(\theta) e^{i\theta},\rho e^{i\phi}\right)\rho d\rho d\phi.$$ By exploring ideas in \eqref{eq2.12},  direct computations yield 
		\begin{align}\label{4-3}
			d_{r_1}F_{1,3}(0, 0)[f_1](\theta)&=\frac{f_1(\theta)}{a_1}\int_{0}^{2\pi}\int_{0}^{a_2 }\nabla G_H\left(a_1  e^{i\theta},\rho e^{i\phi}\right)\cdot e^{i\theta} \rho\,d\rho\,d\phi\\
			&=\frac{f_1(\theta)}{a_1} \partial_{x_1} \left(\int_{\mathbb{R}^2}G_H(x,y)1_{B_{a_2}(0)}(y)dy\right)\bigg|_{x=(a_1,0)}\nonumber\\
			&=:\frac{f_1(\theta)}{a_1}\partial_{x_1} \Psi_2(a_1,0).\nonumber
		\end{align}
		 It is obvious that $\Psi_2$ is radial and satisfies the following ordinary differential equation
		\begin{equation*}
			-\frac{1}{\rho}\partial_\rho\left(\frac{\rho}{1+\rho^2/h^2}\partial_\rho\Psi_2\right)= 1_{\{\rho<a_2\}},
		\end{equation*}
		which can be solved by
		\begin{equation*}
			\Psi_2(\rho)=-\int_{0}^{\rho}\frac{h^2+s^2}{h^2s}\int_{0}^{\min\left\{s,a_2\right\}}\tau d\tau ds+\text{constant},
		\end{equation*}
		and gives
		\begin{equation}\label{4-4}
			\partial_{x_1}\Psi_2(a_1,0)=-\frac{h^2+a_1^2}{2h^2 a_1}a_2^2.
		\end{equation}
		The results from Lemma \ref{lem2.2}, combined with the definition of $\mathcal{F}_1$ in \eqref{eq4-1} and the identities in \eqref{4-3} and \eqref{4-4}, allow us to derive the expression for $d_{r_1}\mathcal{F}_1(\Om, 0, 0)[f_1]$. Consequently, this leads to the expression for the element in the upper left corner of the matrix $M_n(\Om, a_1, a_2)$.
		\smallskip
		
		Next, we compute $d_{r_2}\mathcal{F}_1(\Om, 0, 0)[f_2]$. By the expression of  $\mathcal{F}_1$,   we  find 
		\begin{align*}
			d_{r_2}\mathcal{F}_1(\Om, 0, 0)[f_2](\theta)=-\partial_\theta \int_{0}^{2\pi} G_H\left(a_1  e^{i\theta}, a_2 e^{i\phi}\right)f_2(\phi) d\phi.
		\end{align*}
		Using ideas similar to \eqref{eq2.9}--\eqref{2.11}, one can verify that 
		\begin{equation}\label{4-5}
			\int_{0}^{2\pi} G_H\left(a_1  e^{i\theta}, a_2 e^{i\phi}\right)f_2(\phi) d\phi= \sum_{n=1}^{\infty}\frac{a_1a_2I'_{nm}\left(\frac{nma_2}{h}\right)K'_{nm}\left(\frac{nma_1}{h}\right)}{ h^2}f_n^{(2)}\cos(nm\theta).
		\end{equation}
		The expression for the element in the upper right corner of the matrix $M_n(\Om, a_1, a_2)$  is thus obtained immediately. Similarly, the expression for the element in  the lower left corner of the matrix $M_n(\Om, a_1, a_2)$  can be derived by computing $d_{r_1}\mathcal{F}_2(\Om, 0, 0)[f_1]$ in a similar manner.
		 
		 \bigskip
		 
		It remains to determine the expression for element in the lower right corner of matrix $M_n(\Om, a_1, a_2)$ by calculating  $d_{r_2}\mathcal{F}_2(\Om, 0, 0)[f_2]$. In view of the computations in Lemma \ref{lem2.2}, one finds that only the derivative of $\partial_\theta \int_{0}^{2\pi}\int_{0}^{R_1(\phi)}G_H\left(R_2(\theta) e^{i\theta},\rho e^{i\phi}\right)\rho d\rho d\phi$ with respect to $r_2$ is needed to compute.
		Denote $$ F_{2,3}(r_1, r_2)= \int_{0}^{2\pi}\int_{0}^{R_1(\phi)}G_H\left(R_2(\theta) e^{i\theta},\rho e^{i\phi}\right)\rho d\rho d\phi.$$ Direct computations yield 
		\begin{align*} 
		d_{r_2}F_{2,3}(0, 0)[f_2](\theta)&=\frac{f_2(\theta)}{a_2}\int_{0}^{2\pi}\int_{0}^{a_1 }\nabla G_H\left(a_2  e^{i\theta},\rho e^{i\phi}\right)\cdot e^{i\theta} \rho \,d\rho\,d\phi\\
		&=\frac{f_2(\theta)}{a_2} \partial_{x_1} \left(\int_{\mathbb{R}^2}G_H(x,y)1_{B_{a_1}(0)}(y)dy\right)\bigg|_{x=(a_2,0)}\nonumber\\
		&=:\frac{f_2(\theta)}{a_2}\partial_{x_1} \Psi_3(a_2,0),\nonumber
		\end{align*}
		where $\Psi_3$ is radial and satisfies the following ordinary differential equation
		\begin{equation*}
		-\frac{1}{\rho}\partial_\rho\left(\frac{\rho}{1+\rho^2/h^2}\partial_\rho\Psi_3\right)= 1_{\{\rho<a_1\}},
		\end{equation*}
		which can be solved by
		\begin{equation*}
		\Psi_3(\rho)=-\int_{0}^{\rho}\frac{h^2+s^2}{h^2s}\int_{0}^{\min\left\{s,a_1\right\}}\tau d\tau ds+\text{constant},
		\end{equation*}
		and gives
		\begin{equation*} 
		\partial_{x_1}\Psi_3(a_2,0)=-\frac{h^2+a_2^2}{2h^2}a_2.
		\end{equation*}
		Then we obtain the  expression for  element in the lower right corner of $M_n(\Om, a_1, a_2)$  by combining the above calculation and computations in   Lemma \ref{lem2.2}. 
	\end{proof}

To simplify notations, for $a>0, b>0$ and integer $n\geq 1$, we shall denote 
\begin{equation}\label{4-6}
	\Upsilon(a,b)=\frac{(h^2+a^2)(a^2-b^2)}{2h^2 a^2},\ \ \ \Gamma_n(a, b)=-\frac{a b}{h^2}I'_{n}\left(\frac{n  b}{h}\right)K'_{n}\left(\frac{n a }{h}\right).
\end{equation}
Then the matrix $M_n(\Om, a_1, a_2)$ in Lemma \ref{Linearization-d} can be written as 
\begin{equation}\label{4-7}
	M_n(\Om, a_1, a_2)=	\begin{pmatrix}
		\Om+\Gamma_n(a_1, a_1)- \Upsilon(a_1, a_2) &\Gamma_n(a_1, a_2)\\
		-\Gamma_n(a_1, a_2) &\Om -\Gamma_n(a_2, a_2)
	\end{pmatrix}.
\end{equation}

	To apply Crandall-Rabinowitz theorem, it is requires that the linearized operator $d_{r_1,r_2}\bm{\mathcal{F}}(\Omega,0,0)$ has a one dimensional kernel. In the following proposition, we will select some special $\Om$ such that this one-dimensional properties hold.
	\begin{proposition}\label{prop4-3}
		For $a_1>a_2>0$ and $h>0$, there exists a large constant $M_h(a_1, a_2)>0$ such that, for any integer $n>M_h(a_1,a_2)$, there are two distinct angular velocities 
		\begin{equation*}
			\begin{split}
				\Om^{\pm}_n=&\frac{ \Upsilon(a_1, a_2)-\Gamma_n(a_1, a_1)+\Gamma_n(a_2, a_2)}{2}  \\
				&\pm\frac{\sqrt{\big(\Upsilon(a_1, a_2)-\Gamma_n(a_1, a_1)-\Gamma_n(a_2, a_2) \big)^2-4 \Gamma_n(a_1, a_2)^2}}{2}
			\end{split}
		\end{equation*}
		such that the matrix $M_n(\Om_n^{\pm}, a_1, a_2)$ is singular (i.e., non-invertible). 
		
		In addition, for any $n>M_h(a_1,a_2)$, the sequence $\{\Om^{+}_n\}$ is strictly increasing in $n$ and $\{\Om^{-}_n\}$ is strictly decreasing in $n$.  Furthermore, one has 
		$$\lim_{n\to+\infty} \Om^{+}_n= \Upsilon(a_1, a_2)>0,\ \ \   \lim_{n\to+\infty} \Om^{-}_n=  0.$$
	\end{proposition}
	\begin{proof}
		We need  the following expansions for modified Bessel functions from \cite{KuiOku}. 
		\begin{align}\label{4-8}
			I'_n(n z)&=\frac{1}{\sqrt{2\pi n t}} \frac{e^{n \eta}}{z} \left(1+\frac{\nu_1}{n}+\frac{\nu_2}{n^2}+\cdots\right)\\
			K'_n(n z)&=-\sqrt{\frac{t }{2\pi n }} \frac{e^{-n \eta}}{z} \left(1-\frac{\nu_1}{n}+\frac{\nu_2}{n^2}+\cdots\right),\nonumber
		\end{align}
		where we have used the following notations in \eqref{4-8}: $$ t=\frac{1}{\sqrt{1+z^2}},\ \eta= t^{-1}+\ln \frac{z}{1+t^{-1}},\ \nu_1=\frac{7t^3-9t}{24},\ \nu_2=\frac{-455 t^6+594 t^4-135 t^2}{1152}.$$
		By the above expansions and noticing that $\eta$ is strictly increasing in $z>0$, for $a\geq b>0, h>0$ and large $n$, we find
		\begin{equation}\label{4-9}
			\begin{split}
					0>\Gamma_n(a,b)=\left(e^{-\mu_0(a, b, h)}\right)^n\left(\frac{\mu_1(a, b, h)}{n}+\frac{\mu_2(a, b, h)}{n^2}+O(n^{-3})\right),  
			\end{split}
		\end{equation}
		where $\mu_0(a, b, h)$, $\mu_1(a, b, h)>0$ and $\mu_2(a, b, h)$ are three constants depending on $a, b, h$. In addition,  the constant 
		$$\mu_0(a, b, h)
		\begin{cases}
			=0, &\text{if}\  a=b,\\
			>0, &\text{if}\   a>b.
		\end{cases}$$
		
		\smallskip
		
		  By the expression of $M_n$ in  \eqref{4-7} and direct computation, we get 
		$$\mathrm{det}(M_n(\Om, a_1, a_2))=\Om^2+ \mathcal B_n(a_1, a_2) \Om+\mathcal C_n(a_1,a_2),$$
		where the constants are given by 
		$$ \mathcal B_n(a_1, a_2)=-\Upsilon(a_1, a_2)+\Gamma_n(a_1, a_1)-\Gamma_n(a_2, a_2),$$ and $$\mathcal C_n(a_1,a_2)= \Gamma_n(a_2, a_2)\big( \Upsilon(a_1, a_2)-\Gamma_n(a_1, a_1)\big)+\Gamma_n(a_1, a_2)^2.$$
		We aim to finding $\Om$ such that $\mathrm{det}(M_n(\Om, a_1, a_2))=0$. For this purpose, we compute the discriminant 
		\begin{equation}\label{4-10}
			\begin{split}
						\Delta_n&:= \mathcal B_n(a_1, a_2)^2-4  \mathcal C_n(a_1, a_2)\\
						&=( \Upsilon(a_1, a_2)-\Gamma_n(a_1, a_1)-\Gamma_n(a_2, a_2))^2-4\Gamma_n(a_1, a_2)^2.
			\end{split}
		\end{equation}
		In view of the expansion \eqref{4-9}, we find 
		$$\lim_{n\to +\infty} \Delta_n= \Upsilon(a_1, a_2)^2>0,$$ 
		which implies that $\Delta_n>0$ for $n$ sufficiently large. Therefore, for large $n$, we get two different solutions $\Om_n^{\pm}$ for the equation $\mathrm{det}(M_n(\Om, a_1, a_2))=0$ with 
			\begin{equation*}
			\Om^{\pm}_n=\frac{ \Upsilon(a_1, a_2)-\Gamma_n(a_1, a_1)+\Gamma_n(a_2, a_2)\pm \sqrt{ \Delta_n}}{2}.  
		\end{equation*}
		
		For large $n$, by using the expansion \eqref{4-9} and the fact that $\mu_0(a_1, a_2, h)>0$ due to $a_1>a_2$, we can verify $$\Gamma_n(a_1, a_2) =O\left(\frac{(e^{-\mu_0(a_1, a_2, h)})^n}{n^2}\right)=O\left(\frac{1}{n^3}\right), $$ and hence by Taylor's expansion, we get 
		$$\sqrt{ \Delta_n}=\Upsilon(a_1, a_2)-\Gamma_n(a_1, a_1)-\Gamma_n(a_2, a_2)+O\left(\frac{1}{n^3}\right).$$
		Therefore, by the expression of $\Om_n^{\pm}$, we conclude
			\begin{equation*}
			\begin{split}
			\Om_n^+&=\Upsilon(a_1, a_2)-\Gamma_n(a_1, a_1)+O\left(\frac{1}{n^3}\right),\\
			\Om_n^-&=\Gamma_n(a_2, a_2)+O\left(\frac{1}{n^3}\right).
			\end{split}
		\end{equation*}

		Then, by using \eqref{4-9} again, we obtain the following identities
		\begin{equation*}
			\begin{split}
				\Om_{n+1}^+-\Om_n^+&=\frac{ \mu_1(a_1, a_1, h)}{(n+1)n}+O\left(\frac{1}{n^3}\right),\\
				\Om_{n+1}^--\Om_n^-&=-\frac{ \mu_1(a_2, a_2, h)}{(n+1)n}+O\left(\frac{1}{n^3}\right).
			\end{split}
		\end{equation*}
		This proves the desired monotonic proprieties of $\Om_n^{\pm}$ in $n$ for $n$ sufficiently large. The proof of the proposition is thus finished.
	\end{proof}

	In order to finish our proof of Theorem \ref{thm-doubly} by applying  Crandall-Rabinowitz theorem (Theorem \ref{CR}), we only need to check that the last prerequisites of Crandall-Rabinowitz theorem is satisfied, which will be done in the following proposition.
	\begin{proposition}
		Let $a_1>a_2>0$ and $M_h(a_1,a_2)>0$ be as in Proposition \ref{prop4-3}. Then for all integer $m>M_h(a_1,a_2)$ and $\Om_m^\pm$ defined in Proposition \ref{prop4-3}, the kernel of the linearized operator
		\begin{equation*}
			d_{r_1,r_2}\bm{\mathcal{F}}(\Om_m^\pm,0,0):\mathbb{X}_m\times\mathbb{X}_m\to \mathbb{Y}_m\times\mathbb{Y}_m
		\end{equation*}
		is one-dimensional and generated by $(r_0^1,r_0^2)$, where
		\begin{equation*}
			\begin{pmatrix}
				r_0^1\vspace{0.25em}\\
				r_0^2
			\end{pmatrix}:\theta\mapsto
			\begin{pmatrix}
				\Om_m^\pm-\Gamma_m(a_2,a_2)\vspace{0.25em}\\
				\Gamma_m(a_1,a_2)
			\end{pmatrix}\cos(m\theta).
		\end{equation*}

		Moreover, the range of linearized operator $d_{r_1,r_2}\bm{\mathcal{F}}(\Om_m^\pm,0,0)$ is closed and of co-dimension one.

		Furthermore, the transversality assumption is satisfied, that is, 
		\begin{equation*}
			\partial_\Om d_{r_1,r_2}\bm{\mathcal{F}}(\Om_m^\pm,0,0)[r_0^1,r_0^2]\notin {\rm Range}\left(d_{r_1,r_2}\bm{\mathcal{F}}(\Om_m^\pm,0,0)\right).
		\end{equation*}
	\end{proposition}
	\begin{proof}
		According to the analysis in Proposition \ref{prop4-3}, we can deduce that
		\begin{equation*}
			\det \left(M_m(\Om_m^\pm,a_1,a_2)\right)=0,\quad \det\left(M_{nm}(\Om_m^\pm,a_1,a_2)\right)\neq 0, \text{ for any}\ n\geq 2.
		\end{equation*} 
		In the next, we will give a more precise form about the range of the linearized operator. We equip the Hilbert space $\mathbb{Y}_m\times\mathbb{Y}_m$ with the following $l_2$-scalar product: for any two vectors $(U_1,U_2),$ $(V_1,V_2)\in\mathbb{Y}_m\times\mathbb{Y}_m$ with the following form
		\begin{equation*}
			\begin{split}
				(U_1,U_2)=\left(\sum_{n=1}a_n\sin(nm\theta),\sum_{n=1}b_n\sin(nm\theta)\right),\\
				(V_1,V_2)=\left(\sum_{n=1}c_n\sin(nm\theta),\sum_{n=1}d_n\sin(nm\theta)\right) ,
			\end{split}
		\end{equation*} 
		there holds
		\[
			\langle(U_1,U_2),(V_1,V_2)\rangle:=\sum_{n=1}a_nc_n+b_nd_n.
		\]
		Furthermore, define
		\begin{equation}
			y_0:=\begin{pmatrix}
				\Om_m^\pm-\Gamma_m(a_2,a_2)\\
				-\Gamma_m(a_1,a_2)
			\end{pmatrix}\sin(m\theta).
		\end{equation}
		It is easy to show that 
		\begin{equation}
			(M_m(\Om_m^\pm,a_1,a_2))^T\begin{pmatrix}
				\Om_m^\pm-\Gamma_m(a_2,a_2)\\
				-\Gamma_m(a_1,a_2)
			\end{pmatrix}=0,
		\end{equation}
		where the notation $M^T$ refers to the transpose of the matrix $M$. Then for any $(f_1,f_2)\in\mathbb{X}_m\times\mathbb{X}_m$ given by
		\begin{equation*}
			f_1(\theta)=\sum_{n=1}f_n^{(1)}\cos(nm\theta),\quad f_2(\theta)=\sum_{n=2}f_n^{(2)}\cos(nm\theta),
		\end{equation*}
		for some $f_n^{(1)}$, $f_n^{(2)}\in\mathbb R$, we have 
		\begin{equation*}
			\begin{split}
				\langle d_{r_1,r_2}\bm{\mathcal{F}}(\Om_m^\pm,0,0)[f_1,f_2],y_0 \rangle&=-m\langle M_m(\Om_m^\pm,a_1,a_2)\begin{pmatrix}	f_1^{(1)}\\f_1^{(2)}	\end{pmatrix}, \begin{pmatrix}	\Om_m^\pm-\Gamma_m(a_2,a_2)\\-\Gamma_m(a_1,a_2) \end{pmatrix} \rangle\\
				&=-m\langle \begin{pmatrix}	f_1^{(1)}\\f_1^{(2)}	\end{pmatrix}, M_m^T(\Om_m^\pm,a_1,a_2)\begin{pmatrix}	\Om_m^\pm-\Gamma_m(a_2,a_2)\\-\Gamma_m(a_1,a_2) \end{pmatrix} \rangle\\
				&=0,
			\end{split}
		\end{equation*}
		which implies that $\text{Range}\left(d_{r_1,r_2}\bm{\mathcal{F}}(\Om_m^\pm,0,0)\right)\subset \text{span}^\perp(y_0)$, where
		\[
			\text{span}^\perp(y_0):=\left\{y\in \mathbb{Y}_m\times\mathbb{Y}_m: \langle y,y_0\rangle=0 \right\}.
		\]
		It is clear that $\text{span}^\perp(y_0)$ is a closed sub-space of $\mathbb{Y}_m\times\mathbb{Y}_m$ with co-dimension one. Therefore, if we can show that $\text{Range}\left(d_{r_1,r_2}\bm{\mathcal{F}}(\Om_m^\pm,0,0)\right)\supset  \text{span}^\perp(y_0)$, the range of linearized operator $d_{r_1,r_2}\bm{\mathcal{F}}(\Om_m^\pm,0,0)$ is naturally closed and of co-dimension one.
		To see this, we choose any element $y\in \text{span}^\perp(y_0)$ with the following form
		\begin{equation*}
			y=\begin{pmatrix}
				y_1^{(1)}\\y_1^{(2)}
			\end{pmatrix}\sin (m\theta)
			+\sum_{n=2}\begin{pmatrix}
				y_n^{(1)}\\y_n^{(2)}
			\end{pmatrix}\sin(nm\theta),
		\end{equation*}
		where $(y_1^{(1)},y_1^{(2)})$ satisfies $y_1^{(1)}(\Om_m^\pm-\Gamma_m(a_2,a_2))=y_1^{(2)}\Gamma_m(a_1,a_2)$ and $y_n^{(1)},y_n^{(2)}\in\mathbb{R} $ for $n\geq 2$. 
		Due to the expression of $\Delta_m$ in \eqref{4-10} and the fact that $\Delta_m>0$, as long as $m>M_h(a_1,a_2)$, it implies that 
		\begin{equation}\label{Om-Gamma noteq 0}
			\begin{split}
				\Om_m^\pm-\Gamma_m(a_2,a_2)&=\frac{1}{2}\left(\Upsilon(a_1,a_2)-\Gamma_m(a_2,a_2)-\Gamma_m(a_1,a_1)\right)\pm\frac{1}{2}\sqrt{\Delta_m}\neq 0.
			\end{split}
		\end{equation}
		Combining \eqref{Om-Gamma noteq 0} and the fact $\det (M_{nm}(\Om_m^\pm,a_1,a_2))\neq 0$ for any $n\geq 2$, 
		we can define the function $(f_1,f_2)$ by setting, for all $\theta\in\mathbb{T}$, that
		\begin{equation*}
			\begin{pmatrix}
				f_1\\f_2
			\end{pmatrix}=
			-\frac{1}{m}\begin{pmatrix}
				0\\\frac{y_1^{(2)}}{\Om_m^\pm-\Gamma_m(a_2,a_2)}
			\end{pmatrix}\cos (m\theta)
			-\sum_{n=2}\frac{1}{nm}(M_{nm}(\Om_m^\pm,a_1,a_2))^{-1}\begin{pmatrix}
				y_n^{(1)}\\y_n^{(2)}
			\end{pmatrix}\cos (nm\theta).
		\end{equation*}
		Clearly, we formally have $d_{r_1,r_2}\bm{\mathcal{F}}(\Om_m^\pm,0,0)[f_1,f_2]=y$. Noting that $M_{nm}(\Om_m^\pm,a_1,a_2)=M_{nm}(\Om_{nm}^\pm,a_1,a_2)+(\Om_m^\pm-\Om_{nm}^\pm)I$, where $I$ is the $2\times2$ identity matrix, and $|\Om_m^\pm-\Om_{nm}^\pm|\geq |\Om_m^\pm-\Om_{2m}^\pm|$, we obtain that 
		\begin{equation*}
			\left|(M_{nm}(\Om_m^\pm,a_1,a_2))^{-1}\begin{pmatrix}
				y_n^{(1)}\\y_n^{(2)}
			\end{pmatrix}\right|_{l^\infty}\lesssim |y_n^{(1)}|+|y_n^{(2)} |.
		\end{equation*}
		
		We claim that $(f_1,f_2)\in \mathbb{X}_m\times\mathbb{X}_m$. 
		In fact, we firstly estimate the $L^\infty$ norm of $(f_1,f_2)$, namely
		\begin{equation*}
			\|(f_1,f_2)\|_\infty \lesssim |y_1^{(2)}|+\sum_{n=2}\frac{|y_n^{(1)}|+|y_n^{(2)} |}{n}\lesssim \|y\|_{L^2}<\infty.
		\end{equation*}
		Next, we study the regularity of the derivative of $(f_1,f_2)$. To check that $(f_1',f_2')\in C^{1/2}$, we decompose $(f_1',f_2')$ as follows
		\begin{equation*}
			\begin{split}
				\begin{pmatrix}
					f_1'\vspace{0.25em}\\f_2'
				\end{pmatrix}=&
				\begin{pmatrix}
					0\vspace{0.25em}\\\frac{y_1^{(2)}}{\Om_m^\pm-\Gamma_m(a_2,a_2)}
				\end{pmatrix}\sin(m\theta)
				+\sum_{n=2}(M_{nm}(\Om_m^\pm,a_1,a_2))^{-1}\begin{pmatrix}
					y_n^{(1)}\vspace{0.25em}\\y_n^{(2)}
				\end{pmatrix}\sin (nm\theta)\vspace{0.5em}\\
				=&\begin{pmatrix}
					0\vspace{0.25em}\\\frac{y_1^{(2)}}{\Om_m^\pm-\Gamma_m(a_2,a_2)}
				\end{pmatrix}\sin(m\theta)
				+\sum_{n=2}\widetilde{Q_{n}}\begin{pmatrix}
					y_n^{(1)}\vspace{0.25em}\\y_n^{(2)}
				\end{pmatrix}\sin (nm\theta)+\frac{1}{\Om_{m}^\pm}\sum_{n=2}\begin{pmatrix}
					y_n^{(1)}\vspace{0.25em}\\y_n^{(2)}
				\end{pmatrix}\sin (nm\theta),
			\end{split}
		\end{equation*}
		where $\widetilde{Q_{n}}=(M_{nm}(\Om_m^\pm,a_1,a_2))^{-1}-\frac{1}{\Om_m^\pm}I$.  Similar to the analysis in proof of Proposition \ref{prop3-7}, we obtain that $(f_1',f_2')\in C^{1/2}$.
		

		It remains to show that the transversality assumption holds. Thanks to $\det (M_m(\Om_m^\pm,a_1,a_2))=0$, we obtain that
		\begin{equation*}
			\begin{split}
				&\langle \partial_\Om d_{r_1,r_2}\bm{\mathcal{F}}(\Om_m^\pm,0,0)[r_0^1,r_0^2],y_0\rangle\\
				&=-m \{(\Om_m^\pm-\Gamma_m(a_2,a_2))^2-\Gamma_m^2(a_1,a_2)\}\\
				&=-m(\Om_m^\pm-\Gamma_m(a_2,a_2))(2\Om_m^\pm-\Gamma_m(a_2,a_2)+\Gamma_m(a_1,a_1)-\Upsilon(a_1,a_2)).
			\end{split}
		\end{equation*}
		Thanks to the expression of $\Om_m^\pm$ given in Proposition \ref{prop4-3} and the fact that $\Delta_m>0$, as long as $m>M_h(a_1,a_2)$, it shows that $2\Om_m^\pm-\Gamma_m(a_2,a_2)+\Gamma_m(a_1,a_1)-\Upsilon(a_1,a_2)\neq 0.$ Due to \eqref{Om-Gamma noteq 0}, the term $\Om_m^\pm-\Gamma_m(a_2,a_2)\neq 0$. 
		Therefore, we conclude that 
		\begin{equation*}
			\partial_\Om d_{r_1,r_2}\bm{\mathcal{F}}(\Om_m^\pm,0,0)[r_0^1,r_0^2]\notin {\rm Range } (d_{r_1,r_2}\bm{\mathcal{F}}(\Om_m^\pm,0,0)[r_0^1,r_0^2]).
		\end{equation*}
		The proof is thus finished.
	\end{proof}

\appendix
\section{Crandall-Rabinowitz Theorem}\label{apa}
We recall in this appendix the celebrated Crandall-Rabinowitz theorem   used in the proof of  our main theorems.
 \begin{theorem}\label{CR}
     (Crandall-Rabinowitz \cite{CRANDALL1971321}) Let $\mathbb{X}$ and $\mathbb{Y}$ be two Banach spaces. Consider a neighborhood $\mathbb{V}\subset \mathbb{X}$ of 0 and a function $\mathcal{F}$ such that
     \begin{equation*}
         \mathcal{F}:\mathbb{R}\times \mathbb{V}\to \mathbb{Y}.
     \end{equation*}
Assume that $\mathcal{F}$ satisfies the following assumptions.
\begin{itemize}
    \item [(1)] Existence of trivial branch:
    \begin{equation*}
        \mathcal{F}(\Omega,0)=0,~~\text{for all}~~\Omega\in\mathbb{R}.
    \end{equation*}
    \item [(2)] Regularity: $\mathcal{F}$ is regular in the sense that $\partial_\Omega\mathcal{F}$, $d_x\mathcal{F}$ and $\partial_\Omega d_x\mathcal{F}$ exist and are continuous.
    \item [(3)] Fredholm property: The kernel of $d_x\mathcal{F}(0,0)$ is of dimension one, i.e., there is a $x_0\in\mathbb{V}$ such that 
    \begin{equation*}
        \ker(d_x\mathcal{F}(0,0))={\rm span}\left\{x_0\right\},
    \end{equation*}
    and the ${\rm Range}(d_x\mathcal{F}(0,0))$ is closed and of co-dimension one.
    \item [(4)] Transversality:
    \begin{equation*}
        \partial_\Omega d_x\mathcal{F}(0,0)[x_0]\notin {\rm Range}(d_x\mathcal{F}(0,0)).
    \end{equation*}
\end{itemize}
Then, denoting $\mathcal{X}$ any complement of $\ker(d_x\mathcal{F}(0,0))$ in $\mathbb{X}$, there exist a neighborhood $U$ of $(0,0)$, an interval $(-s_0, s_0)$ with $s_0>0$ a real number and continuous functions
\begin{equation*}
    \Psi:(-s_0, s_0)\to\mathbb{R}\quad\text{and}\quad\Phi:(-s_0,s_0)\to\mathcal{X},
\end{equation*}
such that 
\begin{equation*}
    \Psi(0)=0,~~\Phi(0)=0
\end{equation*}
and 
\begin{equation*}
    \left\{(\Omega,x)\in U:~\mathcal{F}=0\right\}=\left\{(\Psi(s),sx_0+s\Phi(s)):|s|<s_0\right\}\cup\left\{(\Omega,0)\in U\right\}.
\end{equation*}
 \end{theorem}

\section{Continuity of some integral operators}\label{apb}
	We  recall in this appendix the matrix $T_ y$ associated with Green's expansion in Theorem \ref{decom}, which appears in the  Cholesky decomposition is defined by
	\begin{equation*}
		(T_ y)^{-1}(T_ y)^{-t}=K_H( y),
	\end{equation*}
	and recall $R(r)$ appearing in the parameterization of vortex patch defined by
	\begin{equation*}
		R(r)(\theta):=\sqrt{a^2+2r(\theta)}.
	\end{equation*}
	To verify the regularity of the functional $\mathcal F$, we require estimates for certain integral operators. The following lemmas play a crucial role in various parts of the proof of regularity of $\mathcal{F}$.
	\begin{lemma}\label{lem B1}
		Let $\alpha\in[0,1)$ and consider a measurable complex-valued function $K$, defined on $\mathbb{T}\times\mathbb{T} ~\setminus\left\{(\theta,\theta),\theta\in\mathbb{T}\right\}$. Assume that there exists a constant $C>0$ such that
		\begin{equation*}
			|K(\theta,\phi)|\leq\frac{C}{\left|\sin\left(\frac{\theta-\phi}{2}\right)\right|^\alpha},\,~\text{for all}~\, \theta\neq\phi\in\mathbb{T},
		\end{equation*}
		and that $\theta\mapsto K(\theta,\phi)$ is differentiable with
		\begin{equation*}
			|\partial_\theta K(\theta,\phi)|\leq\frac{C}{\left|\sin\left(\frac{\theta-\phi}{2}\right)\right|^{1+\alpha}},~\text{for all}~\, \theta\neq\phi\in\mathbb{T}.
		\end{equation*}
		Then the operator $\mathcal{T}$ defined by
		\begin{equation*}
			f\in L^\infty(\mathbb{T})\mapsto \mathcal{T}f(\theta):=\int_{0}^{2\pi}K(\theta,\phi)f(\phi)d\phi
		\end{equation*}
		belongs to $C^\alpha(\mathbb{T})$ and satisfies
		\begin{equation*}
			\|\mathcal{T}f\|_{C^\alpha}\leq CC'(\alpha)\|f\|_{L^\infty},
		\end{equation*}
		for some constant $C'(\alpha)$.
 	\end{lemma}
 	Based on Lemma \ref{lem B1}, we can derive a series of estimates that facilitate the direct verification of the regularity of the functional. Their proofs can be derived analogously to the proof of Corollary 3.1-3.3 in \cite{hmidi2024uniformly} , utilizing the fact that the positive definite matrix  $T_ y$  has uniform positive lower and upper bounds. The detailed proofs are omitted here.
 
 	\begin{corollary}\label{co1}
		Let $a>0$ and $s\in[0,1]$. Then there exists $\varepsilon_0>0$, such that for any 
		$\|r\|_{C^{1,1/2}}\leq \varepsilon_0$, the operator
		\begin{equation*}
			f\in L^\infty(\mathbb{T})\mapsto\mathcal{T}^1_s f(\theta):=\int_{0}^{2\pi}\ln\left|T_{sR(\phi)e^{i\phi}}(R(\theta)e^{i\theta}-sR(\phi)e^{i\phi})\right|f(\phi)d\phi
		\end{equation*}
		is defined for all $L^\infty$ function and satisfies
		\begin{equation*}
			\|\mathcal{T}_s^1f\|_{C^{1/2}}\leq C\|f\|_{L^\infty},
		\end{equation*}
		where the constant $C>0$ does not dependent on s.
	\end{corollary}
	\begin{corollary}\label{co2}
		Under the assumption of Corollary \ref{co1}, if $\|r_{i}\|_{C^{1,1/2}}\leq\varepsilon_0$ for $i=1,2$, then
		\begin{itemize}
			\item[(1)] Define the operator $\mathcal{T}^2_s$ by
				\begin{equation*}
					g\in L^\infty(\mathbb{T})\mapsto \mathcal{T}^2_sf(\cdot):=\int_{0}^{2\pi}K_s(\cdot,\phi)g(\phi)d\phi,
				\end{equation*}
				where 
				\begin{equation*}
					\begin{split}
						K_s(\theta,\phi):=&\ln|T_{sR(r_1)(\phi)e^{i\phi}}(R(r_1)(\theta)e^{i\theta}-sR(r_1)(\phi)e^{i\phi})|\\
						&-\ln|T_{sR(r_2)(\phi)e^{i\phi}}(R(r_2)(\theta)e^{i\theta}-sR(r_2)(\phi)e^{i\phi})|,
					\end{split}
				\end{equation*}
				then, it holds, for any $\beta<1/2$, that
				\begin{equation*}
					\|\mathcal{T}^2_sg\|_{C^{1/2}}\lesssim\|r_1-r_2\|^\beta_{C^1}\|g\|_{L^\infty}.
				\end{equation*}
			\item[(2)]Define the operator $\mathcal{T}_{f,R}$ by
				\begin{equation*}
					g\in L^\infty(\mathbb{T})\mapsto \mathcal{T}_{f,R}g(\cdot):=\int_{0}^{2\pi}K_{h,R}(\cdot,\phi)g(\phi)d\phi,
				\end{equation*}
				where
				\begin{equation*}
					K_{f,R}(\theta,\phi):=\frac{T'_{R(\phi) e^{i\phi}}T_{R(\phi) e^{i\phi}}(R(\theta)e^{i\theta}-R(\phi) e^{i\phi})}{|T_{R(\phi) e^{i\phi}}(R(\theta)e^{i\theta}-R(\phi) e^{i\phi})|^2}\left(\frac{f(\theta)}{R(\theta)}-\frac{f(\phi)}{R(\phi)}\right),
				\end{equation*}
				then 
				\begin{equation*}
					\|\mathcal{T}_{f,R}g\|_{C^{1/2}}\lesssim\|f\|_{C^{1,1/2}}\|g\|_{L^\infty}
				\end{equation*}
				and
				\begin{equation*}
					\|\mathcal{T}_{f,R(r_1)}g-\mathcal{T}_{f,R(r_2)}g\|_{C^{1/2}}\lesssim\|r_1-r_2\|^\beta_{C^1}\|f\|_{C^{1,1/2}}\|g\|_{L^\infty}.
				\end{equation*}
		\end{itemize}
	\end{corollary}
	\bigskip
	
	\noindent{\bf Acknowledgements:}

	D. Cao, B. Fan and R. Li were supported by National Key R\&D Program of China
	(Grant 2022YFA1005602) and NNSF of China (Grant 12371212). G. Qin was supported by China National Postdoctoral Program for Innovative Talents (No. BX20230009). 
	
	\bigskip
	\noindent{\bf  Data Availability} Data sharing not applicable to this article as no datasets were generated or analysed during the current study.
	
	\bigskip
	\noindent{\bf Declarations}
	
	\bigskip
	\noindent{\bf Conflict of interest}  The authors declare that they have no conflict of interest to this work.
	
	\phantom{s}
	\thispagestyle{empty}
	
	\bibliographystyle{abbrv}
	\bibliography{ref}
	
	\
	
	\
	
	\
	
	\
	\end{document}